\documentclass[12pt,reqno]{amsart}

\usepackage{amsfonts, amsthm, amsmath}
\allowdisplaybreaks[4]

\usepackage{rotating}

\usepackage{tikz}

\usepackage{graphics}

\usepackage{amssymb}

\usepackage{amscd}

\usepackage[latin2]{inputenc}

\usepackage{t1enc}

\usepackage[mathscr]{eucal}

\usepackage{extarrows}

\usepackage{indentfirst}

\usepackage{graphicx}

\usepackage{graphics}

\usepackage{pict2e}

\usepackage{mathrsfs}

\usepackage{enumerate}
\usepackage[pagebackref]{hyperref}
\hypersetup{colorlinks=true}
\usepackage{cite}
\usepackage{color}
\usepackage{epic}
\usepackage{hyperref} 
\numberwithin{equation}{section}
\theoremstyle{plain}

\newtheorem{theorem}{Theorem}[section]

\newtheorem{lemma}[theorem]{Lemma}

\newtheorem{corollary}[theorem]{Corollary}

\newtheorem{claim}[theorem]{Claim}

\newtheorem{proposition}[theorem]{Proposition}

\theoremstyle{definition}

\newtheorem{Def}[theorem]{Definition}

\newtheorem{example}[theorem]{Example}

\newtheorem{conj}[theorem]{Conjecture}

\newtheorem{remark}[theorem]{Remark}

\newtheorem{?}[theorem]{Problem}

\def\Z{\mathbb{Z}}
\def\R{\mathbb{R}}
\def\S{\mathfrak{S}}
\def\D{\mathfrak{D}}

\def\A{\mathfrak{A}}
\def\a{\mathfrak{a}}

\def\exc{\mathsf{exc}}
\def\ai{\mathsf{adi}}
\def\da{\mathsf{da}}
\def\dd{\mathsf{dd}}
\def\des{\mathsf{des}}
\def\asc{\mathsf{asc}}

\def\inv{\mathsf{inv}}
\def\fix{\mathsf{fix}}
\def\fmax{\mathsf{fmax}}
\def\nest{\mathsf{nest}}
\def\cros{\mathsf{cros}}
\def\icr{\mathsf{icr}}
\def\ine{\mathsf{ine}}
\def\wex{\mathsf{wex}}
\def\MAD{\mathsf{MAD}}
\def\valley{\mathsf{valley}}
\def\cvalley{\mathsf{cvalley}}
\def\drop{\mathsf{drop}}
\def\peak{\mathsf{peak}}
\def\cda{\mathsf{cda}}
\def\cdd{\mathsf{cdd}}

\def\Orb{\mathsf{Orb}}
\def\cab{31\text{-}2}
\def\bca{2\text{-}31}
\def\acb{13\text{-}2}
\def\bac{2\text{-}13}
\def\st{\textrm{st}}
\def\fl{\mathsf{fl}}

\def\stat{\mathsf{stat}}

\def\Dyck{\text{Dyck}}

\begin{document}
\title[Gamma expansions, pattern avoidance and the $(-1)$-phenomenon]{$(q,t)$-Catalan numbers:
gamma expansions, pattern avoidance and the $(-1)$-phenomenon}

\author[S. Fu]{Shishuo Fu}
\address[Shishuo Fu]{College of Mathematics and Statistics, Chongqing University, Huxi campus LD506, Chongqing 401331, P.R. China}
\email{fsshuo@cqu.edu.cn}

\author[D. Tang]{Dazhao Tang}

\address[Dazhao Tang]{College of Mathematics and Statistics, Chongqing University, Huxi campus LD206, Chongqing 401331, P.R. China}
\email{dazhaotang@sina.com}

\author[B. Han]{Bin Han}

\address[Bin Han]{Univ Lyon, Universit\'e Claude Bernard Lyon 1, CNRS UMR 5208, Institut Camille Jordan, 43 blvd. du 11 novembre 1918, F-69622 Villeurbanne cedex, France}
\email{han@math.univ-lyon1.fr}

\author[J. Zeng]{Jiang Zeng}

\address[Jiang Zeng]{Univ Lyon, Universit\'e Claude Bernard Lyon 1, CNRS UMR 5208, Institut Camille Jordan, 43 blvd. du 11 novembre 1918, F-69622 Villeurbanne cedex, France}
\email{zeng@math.univ-lyon1.fr}

\date{\today}

\begin{abstract}
The aim of this paper is two-fold. We first prove several new interpretations of a kind of $(q,t)$-Catalan numbers along with their corresponding $\gamma$-expansions using pattern avoiding permutations. Secondly, we give a complete characterization of certain $(-1)$-phenomenon for each subset of permutations avoiding a single pattern of length three, and discuss their $q$-analogues utilizing the newly obtained $q$-$\gamma$-expansions, as well as the continued fraction of a quint-variate generating function due to Shin and the fourth author. Moreover, we enumerate the alternating permutations avoiding simultaneously two patterns, namely $(2413,3142)$ and $(1342,2431)$, of length four, and consider such $(-1)$-phenomenon for these two subsets as well.
\end{abstract}



\maketitle

\tableofcontents

\section{Introduction}\label{sec1: Intro}
%

One of the simple and fertile characterizations of \emph{Catalan numbers} $C_n=\frac{1}{n+1}{2n\choose n}$ is  the following Stieltjes continued fraction expansion (cf. \cite{BCS08,Fl80})
\begin{align*}
\sum_{n=0}^\infty C_n z^n=
\cfrac{1}{
1-\cfrac{z}{1-\cfrac{z}{
\ddots}}}.
\end{align*}
 In this paper we
define  the  $(q,t)$-Catalan  numbers $C_n(t,q)$ as the Taylor coefficients in the following continued fraction expansion
\begin{equation}\label{def-q-nara}
\sum_{n=0}^\infty C_n(t,q) z^n=
\cfrac{1}{
1-\cfrac{z}{1-\cfrac{tz}{
\cfrac{\ddots}{1-\cfrac{q^{k-1}z}{1-\cfrac{tq^{k-1}z}{\ddots}
}}}}}.
\end{equation}
When we take $q=1$ in \eqref{def-q-nara}, the right-hand side reduces to the continued fraction expansion for the generating function of the \emph{Narayana polynomials} $C_n(t,1)$, see \cite{BCS08}, and for $t=1$ we recover the classical $q$-Catalan numbers of Carlitz-Riordan~\cite{CR64}. Blanco and Petersen~\cite{BP14} considered a related $(q,t)$-analog of Catalan numbers $\Dyck(n; t,q)$.
Indeed, by comparing the continued fraction~\eqref{def-q-nara} with that in~\cite[Proposition~2.6]{BP14},  we have
 \begin{align}\label{comb2-q-nara}
\Dyck(n; t,q)=C_n(tq,q^2).
\end{align}
Choosing continued fraction as the definition for $C_n(t,q)$ will facilitate us later in some generatingfunctional proofs. More importantly, the combinatorial constructions behind the proof of Theorem~\ref{qnara} originated from the fundamental works of Flajolet~\cite{Fl80} for the lattice path interpretation of the formal continued fractions, and two bijections between sets of certain weighted  \emph{Motzkin paths} and permutations  due to Fran\c con-Viennot \cite{FV}, and Foata-Zeilberger \cite{FZ}, respectively, see also \cite{CSZ, SZ10, SZ12}.

The first goal of this paper is to establish new combinatorial interpretations for $C_n(t,q)$, as well as their corresponding $\gamma$-expansions, using pattern avoiding permutations, which we define now.
Denote by $\S_n$ the set of permutations of length $n$. Given two permutations $\pi\in\S_n$ and $p\in\S_k,\:k\leq n$, we say that \emph{$\pi$ avoids the pattern $p$} if there does not exist a set of indices $1\le i_1<i_2<\cdots<i_k\le n$  such that the subsequence $\pi(i_1)\pi(i_2)\cdots\pi(i_k)$ of $\pi$ is order-isomorphic to $p$. For example, the permutation $15324$ avoids $231$. The set of permutations of length $n$ that avoid patterns $p_1,p_2,\cdots,p_m$ is denoted as $\S_n(p_1,p_2,\cdots,p_m)$.

A polynomial $f(x)=\sum_{i=0}^{n} a_i x^i\in\R[x]$ is called \emph{palindromic} if $a_i=a_{n-i}$ for $0\le i \le n$. Clearly $\{x^i(1+x)^{n-2i}\}_{i=0}^{\lfloor n/2 \rfloor}$ form a basis for the vector space of all palindromic polynomials of degree no greater than $n$. We call the unique expansion
$$
f(x)=\sum_{i=0}^{\lfloor n/2 \rfloor}\gamma_ix^i(1+x)^{n-2i}
$$
 the \emph{$\gamma$-expansion} for $f(x)$.


%

Now we give the first three main results of this paper, with the definitions of permutation statistics, sets and permutation patterns postponed to the next section.


\begin{theorem}\label{qnara}
The $(q,t)$-Catalan numbers $C_n(t, q)$ have the following ten interpretations
\begin{align*}
C_n(t, q)=\sum_{\pi\in\S_n(\tau)}t^{\des\:\pi}q^{\stat\:\pi},
\end{align*}
with $\tau$ being a pattern of length $3$, and $\stat$ being a permutation statistic. Ten choices for the pair $(\tau,\stat)$ are listed in Table~\ref{ten}.
\begin{table}[tbp]\caption{Ten choices for $(\tau,\stat)$}\label{ten}
\centering
\begin{tabular}{|c|c|c||c|c|c|}
\hline
\# & $\tau$ & $\stat$ & \# & $\tau$ & $\stat$\\
\hline
1 & $231$ & $\acb$ & 6 & $132$ & $\bca$\\
2 & $231$ & $\ai^*$ & 7 & $231$ & $\cab$\\
3 & $312$ & $\bac$ & 8 & $312$ & $\bca$\\
4 & $312$ & $\ai$ & 9 & $213$ & $\acb$\\
5 & $213$ & $\cab$ & 10 & $132$ & $\bac$\\
\hline
\end{tabular}
\end{table}
\end{theorem}

The last four interpretations for the $\gamma$-coefficients given in the next theorem correspond in a natural way to the interpretaions labelled as entries $1,3,5,6$ in Table~\ref{ten}.
%
%

\begin{theorem}\label{thm:des-ai-qgamma}
For  $n\geq1$, the following $\gamma$-expansions formula holds true
\begin{align}\label{qgamma-unify}
C_n(t,q)=\sum_{k=0}^{\lfloor \frac{n-1}{2}\rfloor}
\gamma_{n,k}(q)t^k(1+t)^{n-1-2k},
\end{align}
where
\begin{align}
\gamma_{n,k}(q)&=\sum_{\pi\in\widehat{\S}_{n,k}(321)}q^{\inv\:\pi-\exc\:\pi}\label{eq:Lin-Fu}\\
&=\sum_{\pi\in\widetilde{\S}_{n,k}(213)}q^{(\cab)\:\pi}
=\sum_{\pi\in\widetilde{\S}_{n,k}(312)}q^{(\bac)\:\pi}\label{eq:213:312}\\
&=\sum_{\pi\in\widetilde{\S}_{n,k}(132)}q^{(\bca)\:\pi}
=\sum_{\pi\in\widetilde{\S}_{n,k}(231)}q^{(\acb)\:\pi}\label{eq:231:132}.
\end{align}
\end{theorem}
\begin{remark}
Eq. \eqref{qgamma-unify} with interpretation \eqref{eq:Lin-Fu} is due to Lin and Fu~\cite{LF17}. Moreover, Blanco and Petersen~\cite{BP14} also obtained a $\gamma$-expansion formula for $C_n(tq,q^2)$, which should yield another interpretation for the $\gamma$-coefficients.
\end{remark}
We will also prove the following variation of Theorems~\ref{qnara} and \ref{thm:des-ai-qgamma}.

\begin{theorem}\label{thm:des-ai-qgamma-new} We have
\begin{align}
\label{213--132}
\sum_{\pi\in \S_{n}(213)}t^{\des\:\pi}q^{\ai\:\pi}&=
\sum_{\pi\in \S_{n}(132)}t^{\des\:\pi}q^{\ai^*\:\pi}\\
\label{213--ai}
&=\sum_{k=0}^{\lfloor \frac{n-1}{2}\rfloor}\biggl(\:\sum_{\pi\in\widetilde{\S}_{n,k}(213)}q^{\ai\:\pi}\biggr)t^k(1+t)^{n-1-2k}\\
\label{132--ai}
&=\sum_{k=0}^{\lfloor \frac{n-1}{2}\rfloor}\biggl(\:\sum_{\pi\in\widetilde{\S}_{n,k}(132)}q^{\ai^*\:\pi}\biggr)t^k(1+t)^{n-1-2k}.
\end{align}
\end{theorem}

Our second goal is to derive
new examples of the following \emph{$(-1)$-phenomenon:} for certain combinatorial generating functions for a set of permutations or \emph{derangements}, substituting $-1$ for one of the variables gives an associated generating function over \emph{alternating permutations} in the set. A permutation is said to be alternating (or up-down) if it starts with an ascent and then descents and ascents come in turn. This has been called \emph{reverse alternating} in Stanley's survey~\cite{Sta1} and some of the other literatures but we stick with this convention throughout the paper. We denote by $\A_n$ the set of alternating permutations of length $n$, and by $\A_n(p_1,p_2,\cdots,p_m)$ the set of alternating permutations of length $n$ that avoid patterns $p_1,p_2,\ldots,p_m$.

The rest of the paper is organized as follows. In section~\ref{sec2: Pre} we give most of the definitions and provide the previously known results from the literature, which will be used to prove Theorems~\ref{qnara}, \ref{thm:des-ai-qgamma} and \ref{thm:des-ai-qgamma-new} in section~\ref{sec3: gamma-q-nara}. We consider a variation involving the {\em weak excedance} in section~\ref{sec4: var-q-nara}. Next in section~\ref{sec5: Cat}, we completely determine the existence of $(-1)$-phenomenon for $\S_n(\tau)$, where $\tau$ runs through all permutations in $\S_3$. For example, we have the following $q$-version of the $(-1)$-phenomenon on $\S_n(321)$ concerning $\exc$. Recall~\cite{BP14} that  Carlitz's $q$-Catalan numbers $C_n(q)$ are defined by
\begin{align}
{C}_n(q):=C_n(q,q^2).
\end{align}
Using the  Dyck path interpretation for $C_n(t,q)$ in \eqref{def-q-nara} (see \cite{Fl80}),  we see that $C_n(q)$ is a polynomial of degree $\binom{n}{2}$. For instance,
\begin{align*}
C_0(q) &=C_1(q)=1,\\
C_2(q) &=q+1,\\
C_3(q) &= q^3+q^2+2q+1,\\
C_4(q) &=q^6+q^5+2q^4+3q^3+3q^2+3q+1.
\end{align*}
\begin{theorem}\label{q-nara}
For any $n\ge 1$,
\begin{align}
\label{q-nara-odd}
C_n(-1,q)=\sum_{\pi\in\S_n(321)}(-1)^{\exc\:\pi}q^{\inv\:\pi-\exc\:\pi}&=\begin{cases}0 & \emph{if $n$ is even},\\
(-q)^{\frac{n-1}{2}}C_{\frac{n-1}{2}}(q^2) & \emph{if $n$ is odd},
\end{cases}\\
\label{q-nara-even}
\sum\limits_{\pi\in\D_n(321)}(-1)^{\exc\:\pi}q^{\inv\:\pi}
&=\begin{cases}(-q)^{\frac{n}{2}}C_{\frac{n}{2}}(q^2) & \emph{if $n$ is even},\\
0 & \emph{if $n$ is odd}.
\end{cases}
\end{align}
\end{theorem}

Motivated by Lewis' work \cite{Lew,Lew1,Lew2,Lew3}, many authors \cite{Bon,CCZ,Yan,XY,MW} have studied pattern avoidance on alternating permutations, especially the Wilf-equivalence problem for patterns of length four. As for alternating permutations that avoid two patterns of length four simultaneously, our results in section~\ref{sec6: Sch} concerning $\S_n(2413,3142)$ and $\S_n(1342,2431)$ appear to be new. We close with some remarks to motivate further study along this line.

\section{Definitions and preliminaries}\label{sec2: Pre}
\subsection{Permutation statistics and a proof of Theorem~\ref{q-nara}}

We follow \cite{CSZ, SZ10, SZ12} for notations and the nomenclature of various permutation statistics. First we recall four classical involutions defined on $\S_n$, namely, the inverse, reverse, complement and the composition of the latter two. For $\pi\in\S_n$,
\begin{align*}
\pi^{-1} &:=\pi^{-1}(1)\pi^{-1}(2)\cdots\pi^{-1}(n),\\
\pi^r &:=\pi(n)\cdots\pi(2)\pi(1),\\
\pi^c &:=(n+1-\pi(1))(n+1-\pi(2))\cdots(n+1-\pi(n)),\\
\pi^{rc} &:=(n+1-\pi(n))\cdots(n+1-\pi(2))(n+1-\pi(1)).
\end{align*}
If we use the standard two-line notation to write $\pi$, then $\pi^{-1}$ is obtained by switching the two lines and rearranging the columns to make the first line increasing. For instance, if $\pi=\left(\begin{array}{ccc} 1 & 2 & 3\\ 2 & 3 & 1 \end{array}\right)$, then $\pi^{-1}=\left(\begin{array}{ccc} 1 & 2 & 3\\ 3 & 1 & 2 \end{array}\right)$.

\begin{Def}
Let $\pi=\pi(1)\pi(2)\cdots\pi(n)$ be a permutation, assume $\pi(n+1)=n+1$. A \emph{descent} (resp. an \emph{ascent}) in $\pi$ is a triple $(i,\pi(i),\pi(i+1))$ such that $i\in [n]$ and $\pi(i)>\pi(i+1)$ (resp. $\pi(i)<\pi(i+1)$). Here $\pi(i)$ is called the \emph{descent top} (resp. the \emph{ascent bottom}) and $\pi(i+1)$ is called the \emph{descent bottom} (resp. the \emph{ascent top}). An \emph{excedance} (resp. a \emph{nonexcedance}) in $\pi$ is a pair $(i,\pi(i))$ such that $i\in [n]$ and $\pi(i)>i$ (resp. $\pi(i)\le i$). Here $i$ is called the \emph{excedance bottom} (resp. the \emph{nonexcedance top}) and $\pi(i)$ is called the \emph{excedance top} (resp. the \emph{nonexcedance bottom}). The \emph{number of descents} (resp. the \emph{number of ascents}) in $\pi$ is denoted by $\des\:\pi$ (resp. $\asc\:\pi$), and the \emph{number of excedances} in $\pi$ is denoted by $\exc\:\pi$.
\end{Def}

Next, we introduce four statistics involving three consecutive letters in $\pi$.
\begin{Def}
For $\pi\in\S_n$, let $\pi(0)=\pi(n+1)=0$. Then any entry $\pi(i)$ ($i\in[n]$) can be classified according to one of the four cases:
\begin{itemize}
\item a \emph{peak} if $\pi(i-1)<\pi(i)$ and $\pi(i)>\pi(i+1)$;
\item a \emph{valley} if $\pi(i-1)>\pi(i)$ and $\pi(i)<\pi(i+1)$;
\item a \emph{double ascent} if $\pi(i-1)<\pi(i)$ and $\pi(i)<\pi(i+1)$;
\item a \emph{double descent} if $\pi(i-1)>\pi(i)$ and $\pi(i)>\pi(i+1)$.
\end{itemize}
\end{Def}
Let $\peak\:\pi$ (resp. $\valley\:\pi$, $\da\:\pi$, $\dd\:\pi$) count the number of peaks (resp. valleys, double ascents, double descents) in $\pi$, and define
\begin{align*}
\widetilde{\S}_{n,k}(213)&:=\{\pi\in \S_{n}(213): \dd\:\pi=0, \des\:\pi=k \},\\
\widetilde{\S}_{n,k}(312)&:=\{\pi\in \S_{n}(312): \dd\:\pi=0, \des\:\pi=k \},\\
\widetilde{\S}_{n,k}(132)&:=\{\pi\in \S_{n}(132): \dd^*\:\pi=0, \des\:\pi=k\},\\
\widetilde{\S}_{n,k}(231)&:=\{\pi\in \S_{n}(231): \dd^*\:\pi=0, \des\:\pi=k\}.
\end{align*}
 We emphasize here to avoid any future confusion, that whenever two versions of the same type of statistic exist, such as $\dd$ and $\dd^*$ above, we use the $*$ version to indicate the initial condition $\pi(0)=\pi(n+1)=n+1$, while the non-$*$ version means $\pi(0)=\pi(n+1)=0$, with the only exception being Lemma~\ref{SZ12_lem}.

Besides the patterns mentioned in the introduction, we shall also consider the so-called \emph{vincular patterns} \cite{BS}. The number of occurrences of vincular patterns $\cab$, $\bca$, $\bac$ and $\acb$ in $\pi\in \S_n$ are defined by
\begin{align*}
(\cab)\:\pi&=\#\{(i,j):  i+1<j\le n \text{ and } \pi(i+1)<\pi(j)<\pi(i)\},\\
(\bca)\:\pi&=\#\{(i,j):  j<i<n \text{ and } \pi(i+1)<\pi(j)<\pi(i)\},\\
(\bac)\:\pi&=\#\{(i,j): j<i<n \text{ and } \pi(i)<\pi(j)<\pi(i+1) \},\\
(\acb)\:\pi&=\#\{(i,j): i+1<j\le n \text{ and } \pi(i)<\pi(j)<\pi(i+1) \}.
\end{align*}
Previous notations extend naturally for vincular patterns. For example, we use $\S_n(\cab)$ to denote the set of permutations of length $n$ that avoid the vincular pattern $\cab$.
\begin{Def}
The statistic $\MAD$, the number of {\bf fix}ed points, {\bf w}eak {\bf ex}cedances, {\bf inv}ersion number, {\bf cros}sing number and {\bf i}nverse {\bf cr}ossing number, {\bf nest}ing number and {\bf i}nverse {\bf ne}sting number of $\pi\in\S_n$ are defined by
\begin{align*}
\MAD\:\pi &:=\des\:\pi+(\cab)\:\pi+2\cdot(\bca)\:\pi,\\
\fix\:\pi &:=\sum_{1\leq i\leq n}\chi\Big(\pi(i)=i\Big),\\
\wex\:\pi &:=\exc\:\pi+\fix\:\pi,\\
\inv\:\pi &=\sum_{1\leq i<j\leq n}\chi\Big(\pi(i)>\pi(j)\Big),\\
\cros\:\pi &:=\#\{(i,j):\:i<j\leq\pi(i)<\pi(j)\quad\textrm{or}\quad\pi(i)<\pi(j)<i<j\},\\
\nest\:\pi &:=\#\{(i,j):\:i<j\leq\pi(j)<\pi(i)\quad\textrm{or}\quad\pi(j)<\pi(i)<i<j\},\\
\icr\:\pi &:=\cros\:\pi^{-1},\\
\ine\:\pi &:=\nest\:\pi^{-1},
\end{align*}
where $\chi(A)=1$ if $A$ is true and 0 otherwise.
\end{Def}
\begin{remark}
Although we introduced $\ine$ above, it is really the same statistic as $\nest$. Namely, we have $\ine\:\pi=\nest\:\pi$ for any permutation $\pi$. When $\fix\:\pi=0$ this should be clear from definition. In general, it will suffice to observe that for any $i$ such that $\pi(i)=i$, there are as many $j<i$ with $\pi(j)>i$ as $k>i$ with $\pi(k)<i$.
\end{remark}

For all $1\leq i\leq n$, $\pi(i)$ is called a {\em left-to-right maximum} if $\pi(i)=\textrm{max}\:\{\pi(1),\pi(2),\cdots,\pi(i)\}$. An ascent bottom $\pi(i)$ $(i=1, \cdots, n)$ is called a \emph{foremaximum} of $\pi$ if it is at the same time a left-to-right maximum. Denote the number of foremaxima of $\pi$ by $\fmax\:\pi$.

\begin{Def}[Shin-Zeng]\label{coderange}
A permutation $\pi$ is called {\em coderangement} if $\fmax\:\pi=0$. Let $\D^*_n$ be the subset of $\S_n$ of coderangements.
\end{Def}
Shin and the fourth author \cite[Definition~7]{SZ10} introduced the above linear model of derangements. As suggested by its name, the set of coderangements is equinumerous with the set of derangements.

For the rest of this subsection, we collect all the lemmas that will be useful in later sections, and prove Theorem~\ref{q-nara}.

The Clarke-Steingr\'imsson-Zeng bijection \cite{CSZ} linking $\des$ based statistics with $\exc$ based ones is crucial for our ensuing derivation. It is the composition, say $\varPhi$, of two bijections between $\S_n$ and the set of certain weighted two colored Motzkin paths of length $n$. One bijection is due to Fran\c con and Viennot \cite{FV}, the other is due to Foata and Zeilberger \cite{FZ}. See \cite{CSZ} for a direct description of $\varPhi$ and further details.  The following equidistribution result reveals further properties of $\varPhi$ and is equivalent to \cite[Theorem~8]{SZ10} modulo one application of the inverse map: $\pi\mapsto \pi^{-1}$.
\begin{lemma}[Shin-Zeng]\label{quintuple:des-exc}
For any $n\ge 1$, there is a bijection $\varPhi$ on $\S_n$ such that
\begin{align*}
(\des,\fmax,\cab,\bca,\MAD)\:\pi=(\exc,\fix,\icr,\ine,\inv)\:\varPhi(\pi)\quad \emph{for all }\pi\in\S_n.
\end{align*}
\end{lemma}

Shin and the fourth author \cite{SZ10} deduced the continued fraction expansion for the quint-variate generating function of $\S_n$ with respect to the above statistics.
\begin{lemma}[Shin-Zeng]\label{Lemma:Shin-Zeng}
Let
\begin{align}\label{gf-quintuple}
A_n(x,y,q,p,s)&:=\sum_{\pi\in\S_n}x^{\des\:\pi}y^{\fmax\:\pi}q^{(\cab)\:\pi}p^{(\bca)\:\pi}s^{\MAD\:\pi}\nonumber \\
&=\sum_{\pi\in\S_n}x^{\exc\:\pi}y^{\fix\:\pi}q^{\icr\:\pi}p^{\ine\:\pi}s^{\inv\:\pi}.
\end{align}
Then we have
\begin{align}\label{cf-quintuple}
1+\sum_{n=1}^{\infty}A_n(x,y,q,p,s)z^n=\cfrac{1}{1-b_0z-\cfrac{a_0c_1z^2}{1-b_1z-\cfrac{a_1c_2z^2}{\ddots}}},
\end{align}
where, for $h\ge 0$,
\begin{align*}
a_h=s^{2h+1}[h+1]_{q,ps}, \quad b_h=yp^hs^{2h}+(x+q)s^h[h]_{q,ps},
\end{align*}
and
\begin{align*}
c_h=x[h]_{q,ps},\quad [h]_{u,v}:=(u^{h}-v^{h})/(u-v).
\end{align*}
\end{lemma}
%


 In order to make \eqref{gf-quintuple} suitable for the pattern-avoiding subsets, we have to invoke the following two lemmas.

\begin{lemma}\label{231-321}
For any $n\geq1$,
\begin{align*}
\S_n(\bac)=\S_n(213),\; &\S_n(\cab)=\S_n(312), \\
\S_n(\acb)=\S_n(132),\; &\S_n(\bca)=\S_n(231).
\end{align*}
\end{lemma}
The first equality in Lemma \ref{231-321} was already observed by Claesson \cite{Cla}. The proofs of the remaining three equalities are essentially the same and thus omitted.

\begin{lemma}\label{321-iff-ine=0}
The mapping $\varPhi$ has the property that  $\varPhi(\S_{n}(231))=\S_n(321)$. Consequently, $\pi\in\S_n(321)$ if and only if $\ine\:\pi=0$.
\end{lemma}

\begin{proof}
Since we already know $\varPhi$ is a bijection and that $|\S_n(231)|=|\S_n(321)|=C_n$, it will suffice to show that for any $\sigma\in\S_n(231)$, we have $\pi:=\varPhi(\sigma)\in\S_n(321)$. Suppose on the contrary that $\pi\not\in\S_n(321)$, and we have $\pi(i)>\pi(j)>\pi(k)$ with $1\le i<j<k\le n$. We discuss by two cases:
\begin{itemize}
    \item if $\pi(j)\le j$, then $\pi(k)<\pi(j)\le j<k$ form an inverse nesting of $\pi$;
    \item if $\pi(j)> j$, then $i<j< \pi(j)<\pi(i)$ form an inverse nesting of $\pi$.
\end{itemize}
Therefore in either case, we have $(\bca)\:\sigma=\ine\:\pi>0$, which implies that $\sigma\not\in \S_n(\bca)=\S_n(231)$, a contradiction. Now we see
\begin{align*}
\pi\in\S_n(321) & \text{ if and only if } \varPhi^{-1}(\pi)\in\S_n(231)=\S_n(\bca),\\
& \text{ if and only if } 0=(\bca)\:(\varPhi^{-1}(\pi)) = \ine\: \pi.
\end{align*}
The proof is now completed.
\end{proof}

By the second claim of Lemma~\ref{321-iff-ine=0}, the special case of Lemma~\ref{Lemma:Shin-Zeng} where $p=0$, $q=1$ yields a result of Cheng et al.~\cite[Theorem 7.3]{CEKS}.
\begin{lemma}[Cheng et al.]\label{SZ10p_0}
We have
\begin{align}\label{CEKS}
\sum_{n=0}^{\infty} \biggl(\sum_{\pi\in \S_n(321)} q^{\inv\:\pi}t^{\exc\:\pi}y^{\fix\:\pi}\biggr)z^n
=\cfrac{1}{1-yz-\cfrac{tqz^2}{1-(1+t)qz-\cfrac{tq^3z^2}{1-(1+t)q^2z-\cfrac{tq^5z^2}{\ddots}}}}.
\end{align}
\end{lemma}

We also need a standard \emph{contraction formula} for continued fractions, see \cite[Eqs. (43) and (44)]{SZ10}.
\begin{lemma}[Contraction formula]\label{contra-formula}
There holds
\begin{align*}
\cfrac{1}{
1-\cfrac{c_{1}z}{1-\cfrac{c_{2}z}{
1-\cfrac{c_{3}z}{
1-\cfrac{c_{4}z}
{\ddots}}}}}
&=\cfrac{1}{1-c_{1}z-\cfrac{c_{1}c_{2}z^{2}}{1-(c_{2}+c_{3})z-\cfrac{c_{3}c_{4}z^{2}}{1-(c_4+c_5)z-\cfrac{c_5c_6z^2}{\ddots}}}}\\
&=1+\cfrac{c_{1}z}{1-(c_{1}+c_{2})z-\cfrac{c_{2}c_{3}z^{2}}{1-(c_{3}+c_{4})z-\cfrac{c_{4}c_{5}z^{2}}{1-(c_5+c_6)z-\cfrac{c_6c_7z^2}{\ddots}}}}.
\end{align*}
\end{lemma}

We are now ready to prove Theorem~\ref{q-nara}.

\begin{proof}[Proof of Theorem~\ref{q-nara}]
Taking $(t, y)=(-1/q,1)$ in \eqref{CEKS}, we have by applying the contraction formula
\begin{align*}
1+\sum_{n=1}^{\infty} \biggl(\sum_{\pi\in \S_n(321)}(-1)^{\exc\:\pi} q^{\inv\:\pi-\exc\:\pi}\biggr)z^n=1+\cfrac{z}{1+\cfrac{qz^2}{1+\cfrac{q^3z^2}{1+\cfrac{q^5z^2}{\ddots}}}}.
\end{align*}
We derive \eqref{q-nara-odd}  by comparing this with \eqref{def-q-nara}.

In the same vein, by setting $(t, y)=(-1,0)$ in \eqref{CEKS}, we obtain
\begin{align*}
\sum_{n=0}^{\infty}\biggl(\sum\limits_{\pi\in\D_n(321)}(-1)^{\exc\:\pi}q^{\inv\:\pi}\biggr)z^{n}=\cfrac{1}{1-\cfrac{(-q)z^2}{1-\cfrac{(-q^{3})z^{2}}
{1-\cfrac{(-q^5)z^2}{\ddots}}}}.
\end{align*}
Comparing with \eqref{def-q-nara}, we readily get \eqref{q-nara-even}.
\end{proof}

\subsection{Other combinatorial interpretations of \texorpdfstring{$C_n(q)$}{}}
We can derive several pattern avoiding interpretations for
our $q$-Catalan numbers $C_n(q)$ from $\gamma$-expansions proved in \cite{LF17, Lin17}. Let
\begin{align*}
\widehat{\S}_{n,k}(321) &:=\{\pi\in\S_n(321): \exc\:\pi=k\text{ and if } i<\pi(i),\\
&\phantom{==} \text{ then $i+1$ is a nonexcedance bottom}\}.
\end{align*}
According to this definition, for any $\pi\in\widehat{\S}_{n,k}(321)$, each occurrence of excedance uniquely leads to an occurrence of nonexcedance. So when $n$ is odd, the maximum for $\exc\:\pi$ is achieved at $k=\frac{n-1}{2}$, and in this case, the ``if'' condition becomes ``if and only if''. More precisely, take any $\pi\in\widehat{\S}_{n,\frac{n-1}{2}}(321)$, we have for $1\le i\le n-1$ that the following are equivalent:
\begin{itemize}
    \item $i<\pi(i)$,
    \item $i+1$ is a nonexcedance bottom, and
    \item $\pi(i)-1$ is a nonexcedance top.
\end{itemize}

This analysis shows that $\widehat{\S}_{2n+1,n}(321)$ is exactly the set as described in \cite[Excercise 145]{Sta2}, and therefore it is enumerated by $C_n$.

Interestingly, we find yet another two $q$-$\gamma$-expansions in Lin's work \cite[Theorems~1.2 and 1.4]{Lin17}.
\begin{lemma}[Lin]
For any $n\ge 1$,
\begin{align}
\sum_{\pi\in\S_n(321)}t^{\wex\:\pi}q^{\inv\:\pi} &=\sum_{k=1}^{\lfloor\frac{n+1}{2}\rfloor}\left(\sum_{\pi\in\emph{NDW}_{n,k}(321)}q^{\inv\:\pi}\right)t^k(1+t/q)^{n+1-2k},\label{Lin-eq2}\\
\sum_{\pi\in\D_n(321)}t^{\exc\:\pi}q^{\inv\:\pi} &=\sum_{k=1}^{\lfloor\frac{n}{2}\rfloor}\left(\sum_{\pi\in\emph{NDE}_{n,k}(321)}q^{\inv\:\pi}\right)t^k(1+t)^{n-2k},\label{Lin-eq1}
\end{align}
where
\begin{align*}
\emph{NDW}_{n,k}(321) &:=\{\pi\in\S_{n}(321):\wex\:\pi=k,~\emph{no}~i~\emph{such~that}~\pi(i+1)\geq i+1, i\geq\pi^{-1}(i)\},\\
\emph{NDE}_{n,k}(321) &:=\{\pi\in\D_{n}(321):\exc\:\pi=k,~\emph{no $i$ such that $\pi^{-1}(i)<i<\pi(i)$}\}.
\end{align*}
\end{lemma}

Now we can give the following three alternative interpretations for $C_n(q)$.
\begin{proposition}\label{cnq: three sets}
For any $n\geq1$,
\begin{align*}
C_n(q^2)=q^{-2n}\sum_{\pi\in\widehat{\S}_{2n+1,n}(321)}q^{\inv\:\pi} =q^{-2n}\sum_{\pi\in\emph{NDW}_{2n+1,n+1}(321)}q^{\inv\:\pi}=q^{-n}\sum_{\pi\in\emph{NDE}_{2n,n}(321)}q^{\inv\:\pi}.
\end{align*}
\end{proposition}
\begin{proof}
The first equality involving $\widehat{\S}_{2n+1,n}(321)$ follows from putting $t=-1$ in \eqref{qgamma-unify} and comparing \eqref{eq:Lin-Fu} with \eqref{q-nara-odd}. The second interpretation involving $\text{NDW}_{2n+1,n+1}(321)$ is a result of taking $t=-q$ in \eqref{Lin-eq2}, replacing $\wex\:\pi$ by $2n+1-\exc\:\pi$, and comparing the result with \eqref{q-nara-odd}. The last one follows from \eqref{Lin-eq1} (setting $t=-1$) and \eqref{q-nara-even} similarly.
\end{proof}

\begin{remark}
Two remarks on Proposition \ref{cnq: three sets} are in order. First, as a by-product we note that $\inv\:\pi$ is even for any $\pi\in\widehat{\S}_{2n+1,n}(321)$ (resp. $\pi\in\textrm{NDW}_{2n+1,n+1}(321)$), and $\inv\:\pi$ has the parity of $n$ for any $\pi\in\textrm{NDE}_{2n,n}(321)$. A direct combinatorial explanation of this might be interesting. On the other hand, from a bijective point of view, we note that the second equality above is a natural result of the inverse map $\pi \mapsto \pi^{-1}$, while a bijection deducing the third equality is possible via the two colored Motzkin path \cite{Lin17,LF17}. We leave the details as exercises for motivated readers. Moreover, we note by passing that $|\text{NDE}_{2n,n}(321)|=C_n$ is equivalent to Exercise 151 in \cite{Sta2}.
\end{remark}


\section{Proofs of Theorems \ref{qnara}, \ref{thm:des-ai-qgamma} and \ref{thm:des-ai-qgamma-new}}
\label{sec3: gamma-q-nara}

In order to prove Theorem~\ref{qnara}, we begin with a crucial lemma, which follows from \cite[Eq. (39)]{SZ12}.
\begin{lemma}[Shin-Zeng]\label{SZL1} The following four polynomials are equal
\begin{align*}
&\sum_{\pi\in\S_n}t^{\des\:\pi}p^{(\bac)\:\pi}q^{(\cab)\:\pi}=\sum_{\pi\in\S_n}t^{\des\:\pi}p^{(\cab)\:\pi}q^{(\bac)\:\pi}\\
&=\sum_{\pi\in\S_n}t^{\des\:\pi}p^{(\bca)\:\pi}q^{(\cab)\:\pi}=
\sum_{\pi\in\S_n}t^{\des\:\pi}p^{(\cab)\:\pi}q^{(\bca)\:\pi}.
\end{align*}
\end{lemma}
\begin{proof} Indeed, equation (39) in \cite{SZ12} reads:
\begin{align*}
\sum_{n=0}^{\infty} \biggl(\sum_{\pi\in\S_n}t^{\des\:\pi}p^{(\bac)\:\pi}q^{(\cab)\:\pi}\biggr)z^n&=
\sum_{n=0}^{\infty}\biggl(\sum_{\pi\in\S_n}t^{\des\:\pi}p^{(\bca)\:\pi}q^{(\cab)\:\pi}\biggr)z^n\\
&=\cfrac{1}{1-\cfrac{c_1z}{1-\cfrac{c_2z}{1-\cfrac{c_3z}{\ddots}}}}
\end{align*}
with $c_{2i}=t[i]_{p,q}$ and $c_{2i-1}=[i]_{p,q}$ for $i\geq 1$, where two misprints in \cite{SZ12} are corrected.
The continued fraction shows clearly that the generating function is symmetric in $p$ and $q$.
\end{proof}

Next we give the definition of the statistic \emph{admissible inversion}, which was first introduced by Shareshian and Wachs\cite{SW}.

\begin{Def}\label{ai1}
Let $\pi=\pi(1)\pi(2)\cdots \pi(n)$ be a permutation of $\S_{n}$ and $\pi(0)=\pi(n+1)=0$.  An \emph{admissible inversion} of $\pi$ is an inversion
 pair $(\pi(i),\pi(j))$, i.e., $1\leq i<j\leq n$ and $\pi(i)>\pi(j)$,
  satisfying  either of the following conditions:
\begin{itemize}
 \item $\pi(j)<\pi(j+1)$ or
 \item there is some $l$ such that $i<l<j$ and $\pi(j)>\pi(l)$.
\end{itemize}
\end{Def}
If we apply reverse-complement to this definition, we get the following version, which was also used in \cite[Definition 1]{LZ15}.
\begin{Def}\label{ai2}
Let $\pi=\pi(1)\pi(2)\cdots \pi(n)$ be a permutation of $\S_{n}$ and $\pi(0)=\pi(n+1)=n+1$.
A \emph{star admissible inversion} of $\pi$ is a pair $(\pi(i),\pi(j))$ such that $1\leq i<j\leq n$ and $\pi(i)>\pi(j)$, satisfying either of the following conditions:
\begin{itemize}
 \item $\pi(i-1)<\pi(i)$ or
 \item there is some $l$ such that $i<l<j$ and $\pi(i)<\pi(l)$.
\end{itemize}
\end{Def}

Let $\ai\:\pi$ and $\ai^*\:\pi$ be the numbers of admissible inversions and star admissible inversions of $\pi\in \S_n$, respectively. For example,
if $\pi=231$, then $\ai\:\pi=0$ while $\ai^*\:\pi=2$.

\begin{lemma}\label{ai=pattern}
We have
\begin{align}
\ai\:\pi&=(\bac)\:\pi, \emph{ if }\pi\in \S_n(312),\label{adi1}\\
\ai^*\:\pi&=(\acb)\:\pi, \emph{ if }\pi\in \S_n(231).\label{adi2}
\end{align}
\end{lemma}

\begin{proof}
By Definition~\ref{ai1}, an inversion pair $(\pi(i), \pi(j))$ of a permutation $\pi\in \S_n$ is admissible if and only if  either of the following conditions holds
\begin{itemize}
\item the triple $(\pi(i), \pi(j), \pi(j+1))$ forms a pattern $\bac$ or $3\text{-}12$;
\item the triple $(\pi(i), \pi(l), \pi(j))$ with $i<l<j$ forms a pattern $312$.
\end{itemize}
Thus, if $\pi\in \S_n(312)$, the permutation $\pi$ avoids both $312$ and $3\text{-}12$. This proves \eqref{adi1}. The proof of \eqref{adi2} is similar.
\end{proof}

We have announced ten interpretations for $C_n(t,q)$ using pattern-avoiding permutations, as numbered in Table~\ref{ten} of Theorem~\ref{qnara}. Moreover, the following interpretation of $C_n(t,q)$ using pattern avoiding permutations first appeared in \cite{LF17}
\begin{align}\label{comb1-q-nara}
C_n(t,q)=\sum_{\pi\in\S_n(321)}t^{\exc\:\pi}q^{\inv\:\pi-\exc\:\pi}.
\end{align}
We note that equality~\eqref{comb1-q-nara} follows also from Cheng et al.~\cite[Theorem 7.3]{CEKS} (see Lemma~\ref{SZ10p_0}) by applying a standard contraction formula (see Lemma~\ref{contra-formula}).
For reader's convenience, we use the symbol $i$---$j$ to denote the equality between two interpretations:
$$
\sum_{\pi\in \S_n(\tau_i)}t^{\des\, \pi}q^{\stat_i \pi}=
\sum_{\pi\in \S_n(\tau_j)}t^{\des\, \pi}q^{\stat_j\pi},
$$
for labels $i,j\in\{0, 1, \ldots, 10\}$,
where 0 represents the interpretation  in \eqref{comb1-q-nara}, and
we break down the proof of Theorem~\ref{qnara} as showing equalities represented by all the edges in Figure~\ref{eleven},
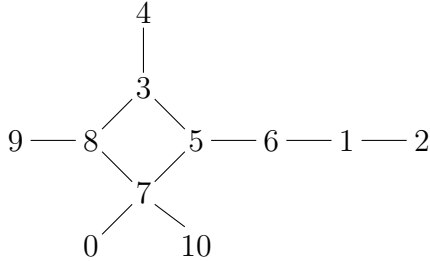
\begin{figure}[htb]
\begin{tikzpicture}[scale=1]
\draw(0,0) node{$9$};
\draw(1,0) node{$8$};
\draw(2.4,0) node{$5$};
\draw(3.4,0) node{$6$};
\draw(4.4,0) node{$1$};
\draw(5.4,0) node{$2$};
\draw(1.7,0.7) node{$3$};
\draw(1.7,1.7) node{$4$};
\draw(1.7,-0.7) node{$7$};
\draw(1,-1.4) node{$0$};
\draw(2.4,-1.4) node{$10$};
\draw (1.15,-.15)--(1.55,-.55);
\draw (1.15,-1.25)--(1.55,-0.85);
\draw (1.15,0.15)--(1.55,0.55);
\draw (1.85,0.55)--(2.25,0.15);
\draw (1.85,-0.55)--(2.25,-0.15);
\draw (1.85,-.85)--(2.25,-1.15);
\draw (0.2,0)--(0.8,0);
\draw (2.6,0)--(3.2,0);
\draw (3.6,0)--(4.2,0);
\draw (4.6,0)--(5.2,0);
\draw (1.7,0.9)--(1.7,1.5);
\end{tikzpicture}
\caption{The proof flowchart for Theorem~\ref{qnara}\label{eleven}}
\end{figure}

\begin{proof}[Proof of Theorem \ref{qnara}]
$\phantom{t}$
\begin{itemize}

	\item 0---7: We make the substitutions $(x,y,q,p,s)=(t/q,1,1,0,q)$
in \eqref{gf-quintuple}, then apply Lemma~\ref{321-iff-ine=0} and the definition of $\MAD$ to obtain
\begin{align*}
\sum_{\pi\in\S_n(321)}t^{\exc\:\pi}q^{\inv\:\pi-\exc\:\pi}
&=\sum_{\pi\in\S_n(231)}t^{\des\:\pi}q^{(\cab)\:\pi}.
\end{align*}
    \item 1---2,\:3---4: These two follow directly from Lemma~\ref{ai=pattern}.
    \item 5---3---7---8:
With Lemma~\ref{231-321} in mind, we set $p=0$ in Lemma~\ref{SZL1} to get
\begin{align*}
\sum_{\pi\in\S_n(213)}t^{\des\:\pi}q^{(\cab)\:\pi}
=\sum_{\pi\in\S_n(312)}t^{\des\:\pi}q^{(\bac)\:\pi}
=\sum_{\pi\in\S_n(231)}t^{\des\:\pi}q^{(\cab)\:\pi}
=\sum_{\pi\in\S_n(312)}t^{\des\:\pi}q^{(\bca)\:\pi}.
\end{align*}
    \item 5---6: The reverse-complement transformation $\pi\mapsto\pi^{rc}$ provides us with
\begin{align*}
\sum_{\pi\in\S_n(213)}t^{\des\:\pi}q^{(\cab)\:\pi}
&=\sum_{\pi\in\S_n(132)}t^{\des\:\pi}q^{(\bca)\:\pi}.
\end{align*}
    \item 8---9,\:7---10,\:6---1: Recall that $C_n(t,q)$ is palindromic in $t$, then we apply the reverse map $\pi\mapsto\pi^r$ to get
\begin{align*}
\sum_{\pi\in\S_n(312)}t^{\des\:\pi}q^{(\bca)\:\pi}
&=\sum_{\pi\in\S_n(312)}t^{n-1-\des\:\pi}q^{(\bca)\:\pi}=\sum_{\pi\in\S_n(213)}t^{\des\:\pi}q^{(\acb)\:\pi},\\
\sum_{\pi\in\S_n(231)}t^{\des\:\pi}q^{(\cab)\:\pi}
&=\sum_{\pi\in\S_n(231)}t^{n-1-\des\:\pi}q^{(\cab)\:\pi}=\sum_{\pi\in\S_n(132)}t^{\des\:\pi}q^{(\bac)\:\pi},\\
\sum_{\pi\in\S_n(132)}t^{\des\:\pi}q^{(\bca)\:\pi}&=\sum_{\pi\in\S_n(132)}t^{n-1-\des\:\pi}q^{(\bca)\:\pi}=\sum_{\pi\in\S_n(231)}t^{\des\:\pi}q^{(\acb)\:\pi}.
\end{align*}
\end{itemize}
By gathering all the equalities above, we complete the proof.
\end{proof}

The proofs of Theorems~\ref{thm:des-ai-qgamma} and \ref{thm:des-ai-qgamma-new} build on several lemmas.

\begin{Def}[MFS-action]
Let $\pi\in\S_n$ with boundary condition $\pi(0)=\pi(n+1)=0$,
 for  any  $x\in[n]$, the {\em$x$-factorization} of $\pi$ reads $\pi=w_1 w_2x w_3 w_4,$ where $w_2$ (resp.~$w_3$) is the maximal contiguous subword immediately to the left (resp.~right) of $x$ whose letters are all larger than $x$.  Following Foata and Strehl~\cite{FSt} we define the action $\varphi_x$ by
$$
\varphi_x(\pi)=w_1 w_3x w_2 w_4.
$$
Note that if $x$ is a double ascent (resp. double descent), then $w_2=\varnothing$ (resp. $w_3=\varnothing$), and if $x$ is a peak then
$w_2=w_3=\varnothing$.
For instance, if $x=3$ and $\pi=28531746\in\S_7$, then $w_1=2,w_2=85,w_3=\varnothing$ and $w_4=1746$.
Thus $\varphi_x(\pi)=23851746$.
Clearly, $\varphi_x$ is an involution acting on $\S_n$ and it is not hard to see that $\varphi_x$ and $\varphi_y$ commute for all $x,y\in[n]$. Br\"and\'en~\cite{Bra08} modified the map $\varphi_x$ to be
\begin{align*}
\varphi'_x(\pi):=
\begin{cases}
\varphi_x(\pi),&\text{if $x$ is not a  valley of $\pi$};\\
\pi,& \text{if $x$ is a valley  of $\pi$.}
\end{cases}
\end{align*}

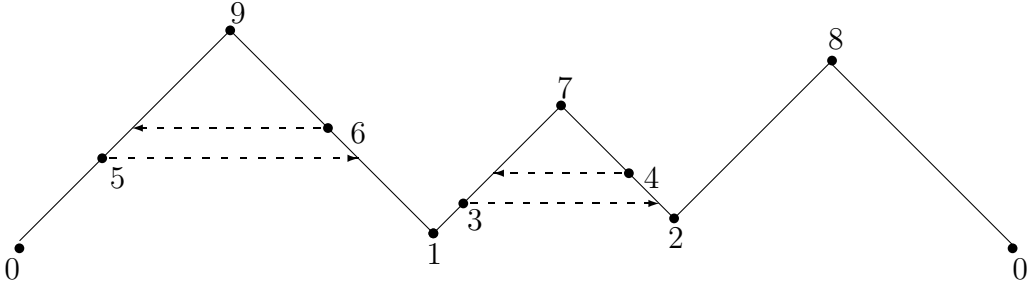
\begin{figure}[htb]
\setlength {\unitlength} {0.8mm}
\begin {picture} (90,40) \setlength {\unitlength} {1mm}
\thinlines
\put(24,8){\dashline{1}(1,0)(25,0)}
\put(50,8){\vector(1,0){0.1}}
\put(-22,14){\dashline{1}(-1,0)(32,0)}
\put(10,14){\vector(1,0){0.1}}
\put(6,18){\dashline{1}(-1,0)(-25,0)}
\put(-20,18){\vector(-1,0){0.1}}
\put(46,12){\dashline{1}(-1,0)(-18,0)}
\put(28,12){\vector(-1,0){0.1}}
\put(-37, -2){$0$}
\put(-35, 2){\circle*{1.3}}
\put(-7,31){\line(-1,-1){28}}
\put(-7,31){\circle*{1.3}}
\put(20,4){\line(-1,1){27}}
\put(-7,32){$9$}
\put(-24,14){\circle*{1.3}}\put(-23,10){$5$}
\put(6,18){\circle*{1.3}}\put(9,16){$6$}
\put(20,4){\circle*{1.3}}
\put(20,4){\circle*{1.3}}\put(19.1,0){$1$}
\put(24,8){\circle*{1.3}}\put(24.5,4.5){$3$}
\put(52,6){\circle*{1.3}}\put(51.2,2){$2$}
\put(20,4){\line(1,1){17}}\put(37,21){\circle*{1.3}}
\put(37,21){\line(1,-1){15}}\put(36.5,22){$7$}
\put(46,12){\circle*{1.3}}\put(48,10){$4$}
\put(52,6){\line(1,1){21}}\put(73,27){\circle*{1.3}}\put(72.5,28.5){$8$}
\put(72.5,27){\line(1,-1){25}}
\put(97, 2){\circle*{1.3}}
\put(97,-2){$0$}
\end{picture}
\caption{MFS-actions on $596137428$ (recall $\pi(0)=\pi(10)=0$)
\label{valhop}}
\end {figure}

See Figure~\ref{valhop} for illustration, where  exchanging
$w_2$ and  $w_3$ in the $x$-factorisation is equivalent to
move $x$  from a double ascent to  a double descent or vice versa. Note that the boundary condition does matter. Take the permutation $596137428$ in Figure~\ref{valhop} as an example.
If $\pi(0)=10$ instead, then $5$ becomes a valley and will be fixed by $\varphi'_5$.

It is clear that $\varphi'_x$'s are involutions and commute. For any subset $S\subseteq[n]$ we can then define the map $\varphi'_S :\S_n\rightarrow\S_n$ by
\begin{align*}
\varphi'_S(\pi)=\prod_{x\in S}\varphi'_x(\pi).
\end{align*}
Hence the group $\Z_2^n$ acts on $\S_n$ via the functions $\varphi'_S$, $S\subseteq[n]$. This action will be called  the {\em Modified Foata--Strehl action} ({\em MFS-action} for short).
\end{Def}

\begin{remark}\label{dual}
The boundary condition $\pi(0)=\pi(n+1)=0$, the definition of $\ai$, and the construction of the MFS-action, are all complement to those used by Lin-Zeng in \cite{LZ15}. When patterns $\{231, 132, \bca, \acb\}$ are concerned, we use Lin-Zeng's version, while for patterns $\{213, 312, \bac, \cab\}$ we use our current version.
\end{remark}
\begin{lemma}\label{lemact}
The statistic $\ai$ is constant on any orbit under the MFS-action.
\end{lemma}
\begin{proof}
Let $\pi\in\S_n$, we aim to show that for each $x\in [n]$, we have $\ai\:\pi=\ai\:\varphi_x'(\pi)$. We discuss by three cases, according to the type of $x$. If $x$ is a peak or a valley of $\pi$, then $\varphi'_x(\pi)=\pi$ and the result is true. If $x$ is a double descent of $\pi$, then the $x$-factorization of $\pi$ is $\pi=w_1w_2xw_3w_4$ with $w_3$ being the empty word, and there are no admissible inversions of $\pi$ formed by $x$ and one letter in $w_2$. As $\varphi'_x(\pi)=w_1w_3xw_2w_4$, there are no inversions of $\varphi'_x(\pi)$ between $w_2$ and $x$. Let $(\pi(i),\pi(j))\notin\{(y, x) : \text{$y$ is a letter in $w_2$}\}$ be a pair in $\pi$ such that $i<j$. We claim that $(\pi(i),\pi(j))$ is an admissible inversion of $\pi$ if and only if it is an  admissible inversion of $\varphi'_x(\pi)$, from which the result follows for this case. Finally, suppose $x$ is a double ascent of $\pi$. Recall that $\varphi_x'$ is an involution, and $x$ is a double ascent of $\pi$ if and only if $x$ is a double descent of $\varphi_x'(\pi)$, so this case follows as well.

Now we prove the claim from last paragraph. For a word $w$, we write $a\in w$ if $a$ is a letter in $w$. To check the claim, we must consider cases depending on whether $\pi(i)$ and $\pi(j)$ belong to $w_1$, $w_2$, $x$, or $w_4$.
 We will show only the case $\pi(i)\in w_2x, \:\pi(j)\in w_4$, other cases are similar. If $(\pi(i),\pi(j))$ is an admissible inversion of $\pi$, then $\pi(i)>\pi(j)<\pi(j+1)$ or $\pi(i)>\pi(j)>\pi(l)$ for some $i<l<j$. Since $\varphi'_x$ does not change the relative order between the letters other than $x$, we see that for the first case and the second case with $\pi(l)\ne x$, $(\pi(i),\pi(j))$ remains an admissible inversion of $\varphi'_x(\pi)$. For the second case with $\pi(l)=x$, we denote $x'$ the first letter of $w_4$. Then $x'<x<\pi(j)<\pi(i)$. Now in view of the triple $(\pi(i),x',\pi(j))$ in $\varphi'_x(\pi)$, we deduce that $(\pi(i),\pi(j))$ is an admissible inversion of $\varphi'_x(\pi)$. To show that, if $(\pi(i),\pi(j))$ is an admissible inversion of $\varphi'(\pi)$ then $(\pi(i),\pi(j))$ is an admissible inversion of $\pi$, is similar and we omit it. This finishes the proof of our claim.
\end{proof}

\begin{lemma}\label{13-2:2-13}
The statistics $(\bca), (\acb), (\bac)$ and $(\cab)$ are constant on any orbit under the MFS-action.
\end{lemma}

\begin{proof}
For $\pi\in \S_n$, when $\pi(0)=\pi(n+1)=n+1$, the cases $(\bca)$ and $(\acb)$ were proved by Br\"and\'en~\cite[Theorem 5.1]{Bra08}. For the case $(\bac)$, let $\pi(0)=\pi(n+1)=0$, and note that
\begin{align*}
(\bac)\:\pi &= \#\{(i,j):\: 1\le i<j<n, \pi(j)<\pi(i)<\pi(j+1)\}\\
&=\#\{(i,j,k):\: 1\le i<j<k\le n, \pi(j)<\pi(i)<\pi(k), \pi(j) \text{ is a valley},\\
&\phantom{==}  \pi(k) \text{ is a peak}, \pi(l) \text{ is neither a valley nor a peak, for }j<l<k\}.
\end{align*}
Here the first equality is by definition. To see the second equality, suppose we are given a pair of indices $(i,j)$ that satisfies the conditions in the first line. Starting with $j$, search to the left looking for the largest valley, say $\pi(j')$, then $i<j'\le j$. Starting with $j+1$, search to the right looking for the smallest peak, say $\pi(k)$, then $j+1\le k\le n$. Now clearly $(i,j',k)$ forms a triple that is counted by the second and third lines above. Conversely, given such a triple, we can uniquely find a pair that is counted by the first line. Hence these two sets are in bijection. Finally, the number of these triples is invariant under the action since $\pi(j)$ and $\pi(k)$ cannot move and neither can $\pi(i)$ hop over the valley $\pi(j)$. A similar argument leads to the case $(\cab)$.
\end{proof}

\begin{lemma}\label{Mfspre}
The MFS-action $\varphi'_S$ is closed on the subsets $\S_n(\tau)$, for $\tau=213,312,132,231$.
\end{lemma}
\begin{proof}
This follows directly from Lemmas \ref{231-321} and \ref{13-2:2-13}.
\end{proof}

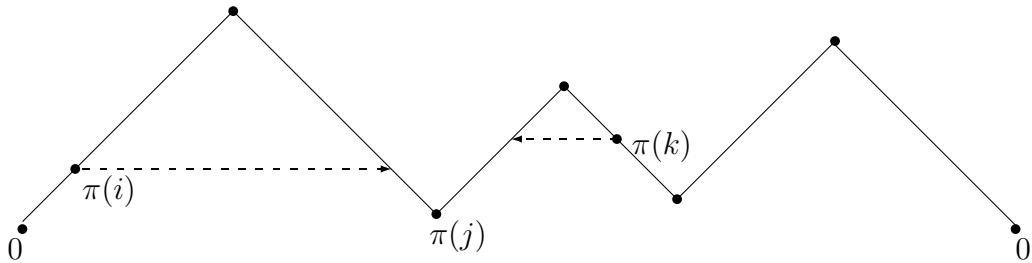
\begin{figure}[htb]
\setlength {\unitlength} {0.8mm}
\begin{picture} (90,40) \setlength {\unitlength} {1mm}
\thinlines

\put(-26,10){\dashline{1}(-1,0)(39,0)}
\put(14,10){\vector(1,0){0.1}}


\put(45,14){\dashline{1}(-1,0)(-13,0)}
\put(30,14){\vector(-1,0){0.1}}

\put(-37, -2){$0$}
\put(-35, 2){\circle*{1.3}}
\put(-7,31){\line(-1,-1){28}}
\put(-7,31){\circle*{1.3}}
\put(20,4){\line(-1,1){27}}
\put(-28,10){\circle*{1.3}}\put(-27,6){$\pi(i)$}
\put(20,4){\circle*{1.3}}\put(19.1,0){$\pi(j)$}
\put(52,6){\circle*{1.3}}
\put(20,4){\line(1,1){17}}\put(37,21){\circle*{1.3}}
\put(37,21){\line(1,-1){15}}
\put(44,14){\circle*{1.3}}\put(46,12){$\pi(k)$}
\put(52,6){\line(1,1){21}}\put(73,27){\circle*{1.3}}
\put(72.5,27){\line(1,-1){25}}
\put(97, 2){\circle*{1.3}}
\put(97,-2){$0$}
\end{picture}
\caption{MFS-actions on pattern avoidance $213$
\label{val213}}
\end{figure}
\begin{proof}[Proof of Theorem \ref{thm:des-ai-qgamma}]
By Lemma~\ref{13-2:2-13}, the statistics tracked by the power of $q$ remain constant inside each orbit under the MFS-action. We first prove the $213$-avoiding case in \eqref{eq:213:312}. For any  permutation $\pi\in\S_{n}$, let $\Orb(\pi)=\{g(\pi): g\in\Z_2^{n}\}$ be the orbit of $\pi$ under the MFS-action. The MFS-action divides the set $\S_{n}$ into disjoint orbits. Moreover, for $\pi\in \S_{n}$, $x$ is a double descent (resp. double ascent) of $\pi$ if and only if $x$ is a double ascent (resp. double descent) of $\varphi'_x(\pi)$. A double descent (resp. double ascent) $x$ of $\pi$ remains a double descent (resp. double ascent) of $\varphi'_y(\pi)$ for any $y\neq x$. Hence, there is a unique permutation in each orbit which has no double descent. Graphically, this is the permutation all of whose active dots hop to the left. Now, let $\bar\pi$ be this unique element in $\Orb(\pi)$, then $\da\:\bar{\pi}=n-\peak\:\bar{\pi}-\valley\:\bar{\pi}$ and $\des\:\bar{\pi}=\peak\:\bar{\pi}-1=\valley\:\bar{\pi}$. And for any other $\pi'\in\Orb(\pi)$, it can be obtained from $\bar{\pi}$ by repeatedly applying $\varphi'_x$ for some double ascent $x$ of $\bar{\pi}$. Each time this happens, $\des$ increases by $1$ and $\da$ decreases by $1$. Thus
\begin{align*}
\sum_{\sigma\in\Orb\:\pi}q^{(\cab)\:\sigma}t^{\des\:\sigma}=
q^{(\cab)\:\bar{\pi}}t^{\des\:\bar{\pi}}(1+t)^{\da\:\bar{\pi}}
=q^{(\cab)\:\bar{\pi}}t^{\des\:\bar{\pi}}(1+t)^{n-2\des\:\bar{\pi}-1}.
\end{align*}
According to Lemma \ref{Mfspre}, by summing over all the orbits that compose together to form $\S_n(213)$, we obtain the $213$-avoiding case \eqref{eq:213:312} immediately. The proofs of the remaining three cases are quite analogous, although one should keep in mind that for the $231$-avoiding and $132$-avoiding cases in \eqref{eq:231:132}, we use Lin-Zeng's version of MFS-action and the initial condition $\pi(0)=\pi(n+1)=n+1$ instead.
\end{proof}
\begin{proof}
[Proof of Theorem \ref{thm:des-ai-qgamma-new}] Clearly the reverse-complement transformation $\pi\mapsto \pi^{rc}$ satisfies $(\des, 213, \ai)\:\pi=(\des, 132, \ai^*)\:\pi^{rc}$, which yields \eqref{213--132} directly. With Lemma \ref{lemact} and Lemma \ref{Mfspre}, we obtain \eqref{213--ai} and \eqref{132--ai} via the MFS-action in a similar fashion as in the proof of Theorem~\ref{thm:des-ai-qgamma}.
\end{proof}

\begin{lemma}\label{213ai}
All inversions in a down-up permutation of odd length are admissible. Moreover, for $\pi\in \widetilde{\S}_{2n+1, n}(213)$, we have
\begin{align}
\label{sum}
\ai\:\pi+\ai\:\pi^r &= 2n^2+n,\\
\label{adi=2}
\ai\:\pi &=2\cdot(3\text{-}12)\:\pi, \\
\label{adir}
\ai\:\pi^r &=(\cab)\:\pi.
\end{align}
\end{lemma}
\begin{proof}
Assume $\pi\in\S_{2n+1}$ is a down-up permutation, and $(\pi(i),\pi(j))$ is an inversion pair. If $j=2n+1$ or $\pi(j)$ is a peak, then $\pi(j-1)<\pi(j)$. Otherwise $\pi(j)$ must be a valley and $j<2n+1$, then we see $\pi(j)<\pi(j+1)$. In both cases, we see that $(\pi(i),\pi(j))$ is indeed admissible. So we have $\inv\:\pi=\ai\:\pi$.

Now suppose $\pi\in \widetilde{\S}_{2n+1, n}(213)$, the initial condition $\pi(2n+2)=0$ and the restriction of having no double descents force $\pi(2n)<\pi(2n+1)$. Since $\pi$ has exactly $n$ descents and no double descents, its descents and ascents must alternate. Moreover, if $2n+1$ is not the first letter of $\pi$ then it will cause a pattern of $213$. In summary, $\pi$ must be a down-up permutation with the first letter being $2n+1$. Also note that the reversal of a down-up permutation of odd length is still a down-up permutation. Therefore we have
$$\ai\:\pi+\ai\:\pi^r = \inv\:\pi+\inv\:\pi^r=\binom{2n+1}{2},$$
where the last equality relies on the simple fact that any pair of letters in $\pi$ is either an inversion pair for $\pi$, or an inversion pair for $\pi^r$. This proves \eqref{sum}.

To get \eqref{adi=2}, we construct a $1\text{-to-}2$ map from triples that form $3\text{-}12$ patterns in $\pi$, to inversion pairs of $\pi$. Suppose $(\pi(i),\pi(j),\pi(j+1))$ is such a triple, then it is mapped to two inversion pairs, namely $(\pi(i),\pi(j))$ and $(\pi(i),\pi(j+1))$. This map is well-defined. Since $\pi$ is down-up, each inversion pair can only be the image under this map for at most one triple. Now it will suffice to show that all inversion pairs in $\pi$ arise in this way. Suppose $(\pi(i),\pi(j))$ is an inversion pair. If $\pi(j)$ is a peak, then $(\pi(i),\pi(j-1),\pi(j))$ forms a $3\text{-}12$ pattern. Otherwise $\pi(j)$ must be a valley, so $\pi(j)<\pi(j+1)$. Now if $\pi(j+1)>\pi(i)$, $(\pi(i),\pi(j),\pi(j+1))$ will form a $213$ pattern which should be avoided by $\pi$. So we must have $\pi(j+1)<\pi(i)$ and $(\pi(i),\pi(j),\pi(j+1))$ forms a $3\text{-}12$ pattern.

Finally for \eqref{adir}, simply note that if $\pi\in\S_n(213)$, then $\pi^r\in\S_n(312)$. Then by Lemma~\ref{ai=pattern} we get $\ai\:\pi^r=(\bac)\:\pi^r=(\cab)\:\pi$ and the proof is completed.
\end{proof}

With the above lemma we obtain another combinatorial interpretation of $C_{n}(q)$.

\begin{proposition}
For any $n\geq 1$,
\begin{align*}
C_{n}(q)=\sum_{\pi\in\widetilde{\S}_{2n+1, n}(213)}q^{n^2-(3\text{-}12)\:\pi}
\end{align*}
\end{proposition}

\begin{proof}
Using the \#5 interpretation for $C_n(t,q)$ and making the same argument as we did in the proof of Theorem~\ref{thm:des-ai-qgamma} concerning the MFS-action, we have
\begin{align*}
C_{2n+1}(t,q)&=\sum_{\pi\in\S_{2n+1}(213)}t^{\des\:\pi}q^{(31\text{-}2)\:\pi}
=\sum_{k=0}^{n}\biggl(\:\sum_{\pi\in\widetilde{\S}_{2n+1,k}(213)}q^{(31\text{-}2)\:\pi}\biggr)t^k(1+t)^{2n-2k}.
\end{align*}
We take $t=-1$ in the above equation and apply \eqref{q-nara-odd} for the left-hand side to get
$$
C_{2n+1}(-1,q)=(-q)^nC_{n}(q^2).
$$
For the right-hand side, only the term with $k=n$ remains. Equating this with the left-hand side and use Lemma~\ref{213ai} to get the desired expression for $C_n(q)$.
\end{proof}

\section{\texorpdfstring{A variant of $q$-Narayana polynomials}{}}\label{sec4: var-q-nara}
\begin{Def}\label{c-stat}
For $\pi\in\S_n$, a value $x=\pi(i)$ ($i\in[n]$) is called
\begin{itemize}
\item a cyclic valley if $i=\pi^{-1}(x)>x$ and $x<\pi(x)$;
\item a double excedance if $i=\pi^{-1}(x)<x$ and $x<\pi(x)$;
\item a drop if $x=\pi(i)<i$.
\end{itemize}
\end{Def}
Let $\cvalley$ (resp. $\cda$, $\drop$) denote the number of cyclic valleys (resp. double excedances, drops) in $\pi$. The following result is due to
Shin-Zeng\cite[Theorem~5]{SZ12}.
\begin{lemma}[Shin-Zeng]\label{SZ12_lem}
There is a bijection $\Upsilon$ on $\S_n$ such that for all $\pi\in \S_n$,
\begin{align*}
(\nest, \cros, \drop, \cda, \cdd, \cvalley, \fix)\:\pi=  (\bca, \cab, \des, \da-\fmax, \dd, \valley, \fmax)\:\Upsilon(\pi),
\end{align*}
where the linear statistics on the right-hand side are defined
with the convention $\pi(0)=0$ and $\pi(n+1)=n+1$ for $\pi\in \S_n$.
\end{lemma}

\begin{theorem}\label{HZthwex}
we have
\begin{align*}
W_n(&t, q):=\sum_{\pi\in\S_n(321)}t^{\wex\:\pi}q^{\inv\:\pi}\\
&=t^n\sum_{\pi\in\S_n(231)}(q/t)^{\des\:\pi}q^{(\cab)\:\pi}=t^n\sum_{\pi\in\S_n(231)}(q/t)^{\des\:\pi}q^{(\acb)\:\pi}=t^n\sum_{\pi\in\S_n(231)}(q/t)^{\des\:\pi}q^{\ai^*\:\pi}\\
&=t^n\sum_{\pi\in\S_n(312)}(q/t)^{\des\:\pi}q^{(\bca)\:\pi}=t^n\sum_{\pi\in\S_n(312)}(q/t)^{\des\:\pi}q^{(\bac)\:\pi}=t^n\sum_{\pi\in\S_n(312)}(q/t)^{\des\:\pi}q^{\ai\:\pi}\\
&=t^n\sum_{\pi\in\S_n(213)}(q/t)^{\des\:\pi}q^{(\cab)\:\pi}=t^n\sum_{\pi\in\S_n(213)}(q/t)^{\des\:\pi}q^{(\acb)\:\pi}\\
&=t^n\sum_{\pi\in\S_n(132)}(q/t)^{\des\:\pi}q^{(\bca)\:\pi}=t^n\sum_{\pi\in\S_n(132)}(q/t)^{\des\:\pi}q^{(\bac)\:\pi}.
\end{align*}
\end{theorem}

\begin{proof}
Since $\drop\:\pi=n-\wex\:\pi$ and $\inv\:\pi=n-\wex\:\pi+\cros\:\pi+2\nest\:\pi$ (\cite[Eq. (40)]{SZ10}), we have
\begin{equation*}
\sum_{\pi\in \S_n(321)} t^{\wex\:\pi}q^{\inv\:\pi}=\sum_{\pi\in \S_n(321)} t^{n-\drop\:\pi}q^{\inv\:\pi}
=t^n\sum_{\pi\in \S_n(321)} ({q}/{t})^{\drop\:\pi}q^{\cros\:\pi}.
\end{equation*}
By Theorem \ref{qnara} and Lemma \ref{SZ12_lem}, we have
$$t^n\sum_{\pi\in \S_n(321)} ({q}/{t})^{\drop\:\pi}q^{\cros\:\pi}=
t^n\sum_{\pi\in \S_n(231)} ({q}/{t})^{\des\:\pi}q^{(\cab)\:\pi}=t^n\sum_{\pi\in \S_n(231)} ({q}/{t})^{\des\:\pi}q^{(\acb)\:\pi},$$
and the remaining equalities follow similarly.
\end{proof}

We can now derive another $q$-$\gamma$-expansion for the joint distribution of $\wex$ and $\inv$ over $\S_n(321)$.
\begin{theorem}\label{qwex}
For any $n\geq 1$,
\begin{align}\label{wex-inv-gamma}
\sum_{\pi\in \S_{n}(321)}t^{\wex\:\pi}q^{\inv\:\pi}=\sum_{k=1}^{\lfloor\frac{n+1}{2}\rfloor}\biggl(\:q^{n-k}\sum_{\pi\in \widetilde{\S}_{n,k-1}(231)}q^{(\acb)\:\pi}\biggr)t^k(1+t/q)^{n+1-2k}.
\end{align}
\end{theorem}
\begin{proof}
By Theorems~\ref{HZthwex} and \ref{thm:des-ai-qgamma},
\begin{align*}
t^n\sum_{\pi\in \S_n(231)} ({q}/{t})^{\des\:\pi}q^{(\acb)\:\pi}
=t^n\sum_{k=0}^{\lfloor\frac{n-1}{2}\rfloor}\left(\sum_{\pi\in\widetilde{\S}_{n, k}(231)} {q^{(\acb)\:\pi}}\right)(q/t)^k(1+q/t)^{n-1-2k}.
\end{align*}

For the right-hand side of above equation,  by shifting $k$ to $k-1$, we get \eqref{wex-inv-gamma}.
\end{proof}

Comparing \eqref{wex-inv-gamma} with \eqref{Lin-eq2}, by utilizing Theorem~\ref{HZthwex} and Theorem \ref{thm:des-ai-qgamma}, we obtain the following equivalent
$q$-analogues of $\gamma$-coefficients with the same arguments in the proof of Theorem \ref{qwex}.

\begin{corollary}\label{gammacoef-othercomb}
There holds
\begin{align*}
\gamma_{n, k}(q)&:=\sum_{\pi\in \mathrm{NDW}_{n,k}(321)}q^{\inv\:\pi}\\
&=q^{n-k}\sum_{\pi\in \widetilde{\S}_{n,k-1}(231)}q^{(\acb)\:\pi}=q^{n-k}\sum_{\pi\in \widetilde{\S}_{n,k-1}(132)}q^{(\bca)\:\pi}\\
&=q^{n-k}\sum_{\pi\in \widetilde{\S}_{n,k-1}(312)}q^{(\bac)\:\pi}=q^{n-k}\sum_{\pi\in \widetilde{\S}_{n,k-1}(213)}q^{(\cab)\:\pi}.
\end{align*}
\end{corollary}

\section{Avoiding one pattern of length three: a complete characterization}\label{sec5: Cat}
\subsection{\texorpdfstring{The $231$-avoding $\des$-case and its $q$-analogues}{}}
The $231$-avoiding alternating permutations were first enumerated by Mansour \cite{Man}:
\begin{align}\label{alt-231=Cn}
|\A_{2n+1}(231)|=|\A_{2n}(231)|=C_n,\quad \text{ for }n\ge 0.
\end{align}
A bijective proof of this fact with further implications was given by Lewis \cite{Lew1}. Indeed, combining \eqref{alt-231=Cn} with exercises 146, 147, 149 and 150 in \cite{Sta2}, utilizing the reverse map, the complement map, as well as the reverse complement map, one get the complete enumerations for all alternating permutations avoiding a single pattern of length three (see Table~\ref{tab:alt-3}).

Recall the {\em standardization} of a word $w$ with $n$ distinct ordered letters, denoted as $\st(w)$, is the unique permutation in $\S_n$ that is order isomorphic to $w$. We say a word $w_1$ is superior to another word $w_2$ and denote as $w_1>w_2$, if for any two letters $l_1\in w_1, l_2\in w_2$, we always have $l_1>l_2$. The following decomposition is crucial for deriving $q$-analogues of the $(-1)$-phenomenon on pattern-avoiding subsets of the coderangements.
\begin{lemma}\label{c-a-b decompose}
Let $P_0(t,q)=Q_0(t,q)=R_1(t,q)=1, P_1(t,q)=Q_1(t,q)=0$, and for $n\ge 2$,
\begin{align*}
P_n(t,q)&:=\sum_{\pi\in\D^*_n(231)}t^{\des\:\pi}q^{(\acb)\:\pi},\\
Q_n(t,q)&:=\sum_{\pi\in\D^*_n(132)}t^{\des\:\pi}q^{(\bca)\:\pi},\\
R_n(t,q)&:=\sum_{\pi\in\D^*_n(213)}t^{\des\:\pi}q^{(\cab)\:\pi}.
\end{align*}
Then for $n\ge 2$,
\begin{align}\label{P-conv}
P_n(t,q) &= \sum_{m=0}^{n-2}tq^{n-m-1}P_m(t,q)C_{n-m-1}(t,q),\\
Q_n(t,q) &= \sum_{m=0}^{n-2}tq^{m}Q_m(t,q)C_{n-m-1}(t,q),\label{Q-conv}\\
R_n(t,q) &= \sum_{m=1}^{n-1}tq^{n-m-1}R_m(t,q)C_{n-m-1}(t,q).\label{R-conv}
\end{align}
\end{lemma}
\begin{proof}
The key observation is that, $\pi\in\D^*_n(231)$ if and only if $\pi=\pi^{(1)}n\pi^{(2)}$, for some subwords $\pi^{(1)}$ and $\pi^{(2)}\neq\varnothing$, such that $\pi^{(1)}\in\D^*_m(231)$ and $\st(\pi^{(2)})\in\S_{n-m-1}(231)$, for some $m, 0\le m\le n-2$, with $\pi^{(2)}>\pi^{(1)}$. Indeed, if $\pi\in\D^*_n(231)$, since $n$ must be a left-to-right maximum of $\pi$, it cannot be at the end of $\pi$, so $\pi^{(2)}\neq\varnothing$. And since $\pi^{(2)}$ is preceded by $n$, it contains no left-to-right maximum, therefore the coderangement restriction does not affect $\pi^{(2)}$ at all, but we do need it to avoid $231$ as a subword of $\pi$. The condition $\pi^{(2)}>\pi^{(1)}$ is to guarantee that $xny$ does not form a $231$ pattern for any $x\in\pi^{(1)}$ and $y\in\pi^{(2)}$. The above discussion shows the ``only if'' part of the claim, the ``if'' part should be clear as well.

Now that the claimed decomposition is justified, we use the appropriate $231$-avoiding interpretation for $C_{n-m-1}(t,q)$ taken from Theorem~\ref{qnara} and examine the change of $\des$ and $(\acb)$ during this decomposition. This should give us \eqref{P-conv}, the proofs of \eqref{Q-conv} and \eqref{R-conv} are similar and thus omitted.
\end{proof}

Now we can derive the following $q$-analogues for the $(-1)$-phenomenon on $\S_n(231)$ concerning $\des$, which parallels Theorem~\ref{q-nara} nicely.
\begin{theorem}\label{thm:q231-des}
For any $n\ge 1$,
\begin{align}
\label{q231-des-odd}
\sum_{\pi\in\S_n(231)}(-1)^{\des\:\pi}q^{(\cab)\:\pi}&=\sum_{\pi\in\S_n(231)}(-1)^{\des\:\pi}q^{(\acb)\:\pi}=\begin{cases}0 & \emph{if $n$ is even},\\
(-q)^{\frac{n-1}{2}}C_{\frac{n-1}{2}}(q^2) & \emph{if $n$ is odd},
\end{cases}\\
\label{q231-des-even1}
\sum_{\pi\in\D^*_n(231)}(-q)^{\des\:\pi}q^{(\cab)\:\pi}&=\begin{cases}(-q)^{\frac{n}{2}}C_{\frac{n}{2}}(q^2) & \emph{if $n$ is even},\\
0 & \emph{if $n$ is odd},
\end{cases}\\
\label{q231-des-even2}
\sum_{\pi\in\D^*_n(231)}(-1)^{\des\:\pi}q^{(\acb)\:\pi}&=\begin{cases}(-1)^{\frac{n}{2}}C^*_{\frac{n}{2}}(q) & \emph{if $n$ is even},\\
0 & \emph{if $n$ is odd},
\end{cases}
\end{align}
where $C^*_n(q):=\sum\limits_{\pi\in\A_{2n}(132)}q^{(\bca)\:\pi}$.
\end{theorem}

\begin{proof}
All we need to do for proving \eqref{q231-des-odd} (resp. \eqref{q231-des-even1}) is take $t=-1$ (resp. $(x,y,q,p,s)=(-1,0,1,0,q)$) in Theorem~\ref{qnara} (resp. \eqref{gf-quintuple}), then apply Theorem~\ref{q-nara}. Next for \eqref{q231-des-even2}, with the decomposition \eqref{P-conv} in mind, we note that $P_{2n+1}(-1,q)=0$ follows from induction on $P_m(-1,q)$ and using \eqref{q-nara-odd} for $C_{n-m-1}(-1,q)$. In the same vein, the even $2n$ case reduces to proving the following identity:
\begin{align}\label{C*-conv}
C^*_n(q) &= \sum_{m=0}^{n-1}q^{3n-3m-2}C^*_m(q)C_{n-m-1}(q^2).
\end{align}
Combining Proposition~\ref{cnq: three sets} and Corollary~\ref{gammacoef-othercomb}, we get the desired interpretation that meshes well with that of $C^*_m(q)$:
\begin{align*}
q^{n-m-1}C_{n-m-1}(q^2) &= \sum_{\pi\in\A_{2n-2m-1}(132)}q^{(\bca)\:\pi}.
\end{align*}
Next we plug this back to \eqref{C*-conv} and decompose permutations in $\A_{2n}(132)$ similarly as in the proof of \eqref{P-conv} to complete the proof.
\end{proof}

The first few values for $C^{*}_n(q)$ are:
\begin{align*}
C^*_0(q) &=C^*_1(q)=1,\\
C^*_2(q) &=2q,\\
C^*_3(q) &=3q^2+2q^4,\\
C^*_4(q) &=4q^3+6q^5+2q^7+2q^9,\\
C^*_5(q) &=5q^4+12q^6+9q^8+8q^{10}+4q^{12}+2q^{14}+2q^{16}.
\end{align*}

\begin{table}[tbp]\caption{The enumeration of $\A_n(\tau)$, for $n\ge 3$ odd and even.}
\label{tab:alt-3}
\centering
\begin{tabular}{ccccccc}
\hline
$\tau$ &123 &132 &213 &231 &312 &321\\
\hline
$\A_{2n+1}(\tau)$ &$C_{n+1}$ &$C_n$ &$C_{n+1}$ &$C_n$ &$C_{n+1}$ &$C_{n+1}$ \\
$\A_{2n}(\tau)$ &$C_{n}$ &$C_n$ &$C_{n}$ &$C_n$ &$C_{n}$ &$C_{n+1}$ \\
\hline
\end{tabular}
\end{table}

\begin{table}[tbp]\caption{The $(-1)$-evaluation over $\S_n(\tau)$ and $\D^*_n(\tau)$ with respect to $\des$.}
\label{tab:avoid-3}
\centering
\begin{tabular}{ccccccc}
\hline
$\des\setminus \tau$ &123 &132 &213 &231 &312 &321\\
\hline
$\S_{2n+1}$ &$\star$ &$(-1)^nC_n$ &$(-1)^nC_n$ &$(-1)^nC_n$ &$(-1)^nC_n$ &$\star$ \\
$\D^*_{2n}$ &$\star$ &$(-1)^nC_n$ &$(-1)^nC_{n-1}$ &$(-1)^nC_n$ &$\star$ &$\star$ \\
\hline
\end{tabular}
\end{table}

\subsection{\texorpdfstring{Other $\des$-cases avoiding one pattern of length three and their $q$-analogues}{}}
In a search for results analogous to Theorems~\ref{thm:q231-des} and \ref{q-nara}, we consider all the remaining subsets that avoid one pattern of length three, and summarize the results in Tables~\ref{tab:avoid-3} and \ref{tab:exc-avoid3}. The $(-1)$-evaluations for the non-$\star$ entries with missing parity are understood to vanish. For instance, $\sum_{\pi\in\S_{2n}(132)}(-1)^{\des\:\pi}=0$. A ``$\star$'' means there is no such phenomenon in this case. Take the top-left $\star$ for example, we put it there to indicate that neither do $\sum_{\pi\in \S_{2n}(123)}(-1)^{\des\:\pi}$ always vanish, nor do we recognize $\sum_{\pi\in \S_{2n+1}(123)}(-1)^{\des\:\pi}$ as a familiar sequence. For all the $\des$-cases, we actually obtain the stronger $q$-versions. We begin by proving three useful lemmas.

\begin{lemma}\label{D*213=nS213}
For any $n\ge 1$,
\begin{align*}
\sum_{\pi\in\D^*_n(213)}t^{\des\:\pi}q^{(\acb)\:\pi}=t\sum_{\pi\in\S_{n-1}(213)}t^{\des\:\pi}q^{(\acb)\:\pi}.
\end{align*}
\end{lemma}
\begin{proof}
It is easy to see from the definition of $\D^*_n$ that $\pi\in\D^*_n(213)$ if and only if $\pi=n\pi'$ with $\pi'\in\S_{n-1}(213)$. Moreover, we note that $\des\:\pi=1+\des\:\pi'$ and $(\acb)\:\pi=(\acb)\:\pi'$. Summing over all the $\pi\in\D^*_n(213)$ completes the proof.
\end{proof}

\begin{lemma}
For $n\ge 1$ and any $\pi\in\S_n$,
\begin{align}\label{idf}
\des\:\pi+(\cab)\:\pi+1 &= \fl\:\pi+(\acb)\:\pi,
\end{align}
where $\fl\:\pi=\pi(1)$ is the {\bf f}irst {\bf l}etter of $\pi$.
\end{lemma}
\begin{proof}
We use induction on $n$. The $n=1$ case holds trivially. Assume \eqref{idf} is true for any permutation with length less than $n$. Let $\pi\in\S_n$ and $i$ be such that $\pi(i)=n$. If $i\in\{1, n\}$, the statement is easily checked. Otherwise we assume $2\le i \le n-1$ and let $\pi'=\pi(1)\cdots\pi(i-1)\pi(i+1)\cdots\pi(n)$.
\begin{itemize}
	\item If $\pi(i-1)<\pi(i+1)$, then $\des\:\pi=\des\:\pi'+1$, $\fl\:\pi=\fl\:\pi'$, and $$(\acb)\:\pi-(\acb)\:\pi'=(\cab)\:\pi-(\cab)\:\pi'+1,$$ where we only need to check the contributions for $\acb$ and $\cab$ coming from the triple with $n$ playing the role of $3$.
	\item $\pi(i-1)>\pi(i+1)$, then $\des\:\pi=\des\:\pi'$, $\fl\:\pi=\fl\:\pi'$, and $$(\acb)\:\pi-(\acb)\:\pi'=(\cab)\:\pi-(\cab)\:\pi'.$$
\end{itemize}
In both cases, we see that \eqref{idf} holds for $n$ as well.
\end{proof}

\begin{lemma}\label{lem:132-fl}
For any $n\ge 2$,
\begin{align}\label{qgamma-fl}
\sum_{\pi\in\D^*_n(132)}t^{\des\:\pi}q^{\fl\:\pi} &= \sum_{k=1}^{\lfloor\frac{n}{2}\rfloor}\biggl(\sum_{\pi\in\overline{\D}^*_{n,k}(132)}q^{\fl\:\pi}\biggr)t^k(1+t)^{n-2k},
\end{align}
where $\overline{\D}^*_{n,k}(132):=\{\pi\in\D^*_n(132): \dd^*\:\pi=1, \des\:\pi=k\}$.
\end{lemma}
\begin{proof}
Since pattern $132$ is concerned here, per Remark~\ref{dual}, we shall use Lin-Zeng's dual version of the MFS-action $\varphi_x$. In addition, we modify it differently in the following way. This new variant of MFS-action is denoted as $\overline{\varphi}_x$.
\begin{align*}
\overline{\varphi}_x(\pi):=
\begin{cases}
\pi,& \text{if $x$ is a valley, a peak, or a left-to-right maximum of $\pi$;}\\
\varphi_x(\pi),&\text{otherwise}.
\end{cases}
\end{align*}

We state without proving the following facts about $\overline{\varphi}_x$, all of which can be verified similarly as for $\varphi'_x$.
\begin{itemize}
	\item $\overline{\varphi}_x$'s are involutions and commute;
	\item the map $\overline{\varphi}_S$ is closed on $\D^*_n(132)$;
	\item for any $\pi\in\D^*_n(132)$ and each $x\in[n]$, $\fl\:\pi=\fl\:\overline{\varphi}_x(\pi)$.
\end{itemize}

Let $\pi\in\D^*_n(132)$. The above facts, together with a similar argument about the orbits under this new MFS-action, tell us that there is a unique permutation in $\Orb(\pi)$ which has exactly one double descent at the first letter (this is due to the definition of coderangements $\D^*$ and the convention that $\pi(0)=\pi(n)=n+1$). Now, let $\bar\pi$ be this unique element in $\Orb(\pi)$, then $\da^*\:\bar{\pi}=n-1-\peak^*\:\bar{\pi}-\valley^*\:\bar{\pi}$ and $\des\:\bar{\pi}=\peak^*\:\bar{\pi}+1=\valley^*\:\bar{\pi}$. Thus
\begin{align*}
\sum_{\sigma\in\Orb\:\pi}t^{\des\:\sigma}q^{\fl\:\sigma}=
q^{\fl\:\bar{\pi}}t^{\des\:\bar{\pi}}(1+t)^{\da^*\:\bar{\pi}}
=q^{\fl\:\bar{\pi}}t^{\des\:\bar{\pi}}(1+t)^{n-2\des\:\bar{\pi}}.
\end{align*}
Summing over all the orbits establishes \eqref{qgamma-fl}.
\end{proof}

Now we are ready to present the $q$-analogues for all the remaining entries shown in Table~\ref{tab:avoid-3}.

\begin{theorem}\label{qnara-des}
For any $n\ge 1$,
\begin{align}
\label{q132-des-odd}
\sum_{\pi\in\S_n(132)}(-1)^{\des\:\pi}q^{(\bca)\:\pi}&=\sum_{\pi\in\S_n(132)}(-1)^{\des\:\pi}q^{(\bac)\:\pi}=\begin{cases}0 & \emph{if $n$ is even},\\
(-q)^{\frac{n-1}{2}}C_{\frac{n-1}{2}}(q^2) & \emph{if $n$ is odd},
\end{cases}\\
\label{q213-des-odd}
\sum_{\pi\in\S_n(213)}(-1)^{\des\:\pi}q^{(\cab)\:\pi}&=\sum_{\pi\in\S_n(213)}(-1)^{\des\:\pi}q^{(\acb)\:\pi}=\begin{cases}0 & \emph{if $n$ is even},\\
(-q)^{\frac{n-1}{2}}C_{\frac{n-1}{2}}(q^2) & \emph{if $n$ is odd},
\end{cases}\\
\label{q312-des-odd}
\sum_{\pi\in\S_n(312)}(-1)^{\des\:\pi}q^{(\bca)\:\pi}&=\sum_{\pi\in\S_n(312)}(-1)^{\des\:\pi}q^{(\bac)\:\pi}=\begin{cases}0 & \emph{if $n$ is even},\\
(-q)^{\frac{n-1}{2}}C_{\frac{n-1}{2}}(q^2) & \emph{if $n$ is odd},
\end{cases}\\
\label{q132-des-even1}
\sum_{\pi\in\D^*_n(132)}(-1)^{\des\:\pi}q^{(\bca)\:\pi}&=\begin{cases}(-1)^{\frac{n}{2}}\widehat{C}_{\frac{n}{2}}(q) & \emph{if $n$ is even},\\
0 & \emph{if $n$ is odd},
\end{cases}\\
\label{q132-des-even2}
\sum_{\pi\in\D^*_n(132)}(-q)^{\des\:\pi}q^{(\cab)\:\pi}&=\begin{cases}(-q)^{\frac{n}{2}}\overline{C}_{\frac{n}{2}}(q) & \emph{if $n$ is even},\\
0 & \emph{if $n$ is odd},
\end{cases}\\
\label{q213-des-even1}
\sum_{\pi\in\D^*_n(213)}(-1)^{\des\:\pi}q^{(\acb)\:\pi}&=\begin{cases}(-1)^{\frac{n}{2}}q^{\frac{n-2}{2}}C_{\frac{n-2}{2}}(q^2) & \emph{if $n$ is even},\\
0 & \emph{if $n$ is odd},
\end{cases}\\
\label{q213-des-even2}
\sum_{\pi\in\D^*_n(213)}(-q)^{\des\:\pi}q^{(\cab)\:\pi}&=\begin{cases}(-1)^{\frac{n}{2}}q^{\frac{3n-4}{2}}C_{\frac{n-2}{2}}(q^2) & \emph{if $n$ is even},\\
0 & \emph{if $n$ is odd},
\end{cases}
\end{align}
where
\begin{align*}
\widehat{C}_n(q):=\sum\limits_{\pi\in\A_{2n}(231)}q^{(\acb)\:\pi}\quad \emph{and} \quad\overline{C}_n(q):=\sum\limits_{\pi\in\A_{2n}(231)}q^{(\bac)\:\pi}.
\end{align*}
\end{theorem}

\begin{proof}
\eqref{q132-des-odd}--\eqref{q312-des-odd} follow directly by taking $t=-1$ in Theorem~\ref{qnara} and applying \eqref{q-nara-odd}. The proof of \eqref{q132-des-even1} parallels that of \eqref{q231-des-even2}, only that we use the decomposition in \eqref{Q-conv} this time. To prove \eqref{q132-des-even2}, we first note that
$$\sum_{\pi\in\D^*_n(132)}(-q)^{\des\:\pi}q^{(\cab)\:\pi}=\sum_{\pi\in\D^*_n(132)}(-1)^{\des\:\pi}q^{\des\:\pi+(\cab)\:\pi}\stackrel{\eqref{idf}}
{=}\sum_{\pi\in\D^*_n(132)}(-1)^{\des\:\pi}q^{\fl\:\pi-1},$$ which gives directly the odd $2n+1$ case in view of the expansion \eqref{qgamma-fl}. For the even $2n$ case, we compute using \eqref{qgamma-fl} again that
$$\sum_{\pi\in\D^*_{2n}(132)}(-1)^{\des\:\pi}q^{\fl\:\pi-1}=(-1)^n\sum_{\pi\in\overline{\D}^*_{2n,n}(132)}q^{\fl\:\pi-1}\stackrel{\eqref{idf}}
{=}(-q)^n\sum_{\pi\in\overline{\D}^*_{2n,n}(132)}q^{(\cab)\:\pi}.$$
Moreover, we note that $\pi\in \overline{\D}^*_{2n,n}(132)$ if and only if $\pi^r\in \A_{2n}(231)$, which implies \eqref{q132-des-even2}.

Finally, \eqref{q213-des-even1} follows from \eqref{q213-des-odd} and Lemma~\ref{D*213=nS213}. In view of the similarity between \eqref{q213-des-even1} and \eqref{q213-des-even2}, it is a straightforward calculation basing on identity \eqref{idf} and that $\fl\:\pi=n$ for any $\pi\in\D^*_n(213)$, as pointed out in the proof of Lemma~\ref{D*213=nS213}.
\end{proof}

The first few values for $\widehat{C}_n(q)$ and $\overline{C}_n(q)$ are:
\begin{align*}
\widehat{C}_0(q) &=\widehat{C}_1(q) =1,\\
\widehat{C}_2(q) &=q+q^2,\\
\widehat{C}_3(q) &=q^2+q^3+q^4+q^5+q^6,\\
\widehat{C}_4(q) &=q^3+q^4+2q^5+2q^6+2q^7+q^8+2q^9+q^{10}+q^{11}+q^{12},\\
\widehat{C}_5(q) &=q^4+q^5+3q^6+3q^7+4q^8+3q^9+5q^{10}+3q^{11}+4q^{12}+3q^{13}+3q^{14}\\
&\phantom{=}+2q^{15}+2q^{16}+2q^{17}+q^{18}+q^{19}+q^{20},\\
\overline{C}_0(q) &=\overline{C}_1(q)=1,\\
\overline{C}_2(q) &=1+q,\\
\overline{C}_3(q) &=1+2q+2q^2,\\
\overline{C}_4(q) &=1+3q+5q^2+5q^3,\\
\overline{C}_5(q) &=1+4q+9q^2+14q^3+14q^4,\\
\overline{C}_6(q) &=1+5q+14q^2+28q^3+42q^4+42q^5.
\end{align*}

The $q$-Catalan numbers $\overline{C}_n(q)$ merit some further investigation for their own sake. First we utilize \eqref{idf} again to get another interpretation for $\overline{C}_n(q)$: $$\sum\limits_{\pi\in\A_{2n}(231)}q^{(\bac)\:\pi}=\sum\limits_{\pi\in\A_{2n}(231)}q^{(\cab)\:\pi^r}\stackrel{\eqref{idf}}{=}q^{-n-1}\sum\limits_{\pi\in\A_{2n}(231)}q^{\fl\:\pi^r}.$$

\begin{Def}
Let $\overline{C}_n(q)=q^{-n-1}\sum\limits_{\pi\in\A_{2n}(231)}q^{\fl\:\pi^r}:=\sum\limits_{k=0}^{n-1}a_{n, k}q^k$, where
\begin{align*}
\a_{n, k}=\{\pi\in \A_{2n}(231): \fl\:\pi^r=n+k+1\}\,\, \text{and} \,\, a_{n, k}=|\a_{n, k}|.
\end{align*}
\end{Def}
The first few examples are~:
\begin{align*}
\a_{1,0}&=\{12\};\\
\a_{2,0}&=\{1423\}\quad\text{and}\quad \a_{2,1}=\{1324\};\\
\a_{3,0}&=\{162534\}, \;\a_{3,1}=\{162435, 132645\}\quad\text{and}\quad
\a_{3,2}=\{132546, 152436\}.
\end{align*}

Recall that the  ballot numbers $f(n,k)$ satisfy  (see \cite{Aig08, CZ15}) the recurrence relation
\begin{align}\label{eq:ballot}
f(n,k)=f(n,k-1)+f(n-1, k),\quad (n,k\geq 0),
\end{align}
where $f(n,k)=0$ if $n<k$ and $f(0,0)=1$, and have the explicit formula
\begin{align*}
f(n,k)=\frac{n-k+1}{n+1}\binom{n+k}{k},\quad (n\geq k\geq 0).
\end{align*}
With the initial values $a_{1, 0}=a_{2, 0}=a_{2, 1}=1$, and comparing \eqref{eq:ballot} and \eqref{balrec}, we establish the following connection.
\begin{proposition}
For $0\leq k\leq n-1$,
\begin{align*}
a_{n, k}=f(n-1,k)=\frac{n-k}{n}\binom{n-1+k}{k}.
\end{align*}
\end{proposition}
\begin{proof}
For $n,k\geq 0$ let $a_{0,0}=1$ and $a_{n, k}=0$ if $k\geq n$ or $k<0$.
It suffices to prove
the following  recurrence relation for $a_{n, k}$:
\begin{equation}\label{balrec}
a_{n+1, k}=a_{n+1, k-1}+a_{n,k}.
\end{equation}
First note two useful facts for any $\pi\in\A_{2n}(231)$.
\begin{enumerate}[a)]
\item $\fl\:\pi=1$, since otherwise $(\pi(1), \pi(2), 1)$ will form a $231$ pattern.
\item $\pi(1)<\pi(3)<\cdots<\pi(2n-1)$, i.e., the valleys of $\pi$ form an increasing subsequence.
\end{enumerate}
Due to fact a), we can assume $\fl\:\pi^r=\pi(2n)>1$. Now we decompose $\a_{n,k}$ as the union of two disjoint subsets:
\begin{align*}
\a^{p}_{n,k}&:=\{\pi\in\a_{n,k} : \pi(2n)-1 \text{ is a peak}\},\\
\a^{v}_{n,k}&:=\{\pi\in\a_{n,k} : \pi(2n)-1 \text{ is a valley}\}.
\end{align*}
We proceed to show that $|\a^{p}_{n+1,k}|=|\a_{n+1,k-1}|$ and $|\a^{v}_{n+1,k}|=|\a_{n,k}|$ via two bijections $\alpha: \a^{p}_{n+1,k}\rightarrow\a_{n+1,k-1}$ and $\beta: \a^{v}_{n+1,k}\rightarrow \a_{n,k}$, and thus proving \eqref{balrec}.

The first map $\alpha$ is relatively easier. For any $\pi\in\a^{p}_{n+1,k}$, we get its image $\alpha(\pi)$ by switching the position of two peaks $\pi(2n)$ and $\pi(2n)-1$. A moment of reflection should reveal that $\alpha: \a^{p}_{n+1,k}\rightarrow\a_{n+1,k-1}$ is indeed well-defined and bijective.

We have to lay some ground work for the second map $\beta: \a^{v}_{n+1,k}\rightarrow \a_{n,k}$. The key observation is on the last three letters. We claim that for any $\pi\in\a^{v}_{n+1,k}$,
\begin{align}\label{last3}
\pi(2n)=\pi(2n+2)+1,\;\pi(2n+1)=\pi(2n+2)-1.
\end{align}
First we see $2n+2\neq \pi(2n+2)$, since otherwise $2n+1=\pi(2n+2)-1$ cannot be a valley. So $2n+2$ must be a non-terminal peak. Now notice that $2n+1$ cannot appear to the left of $2n+2$, otherwise it will cause a $231$ pattern. It must also be a peak, since there are no other letters larger than it except for $2n+2$. If $2n+1=\pi(2n+2)$ is the last peak, then $2n$ being a valley forces $(\pi(2n),\pi(2n+1),\pi(2n+2))=(2n+2,2n,2n+1)$, which means \eqref{last3} holds true. Otherwise $2n+1$ is a non-terminal peak and we consider $2n$ next. This deduction must end in finitely many steps since the total number of peaks is $n$ (and finite). At this ending moment we find some $m$ as the last peak, and $2n+2, 2n+1, \ldots, m+1$ are all peaks decreasingly ordered to its left, then $m-1$ being a valley, together with fact b) force us to have \eqref{last3} again. So the claim is proved.

The definitions and validity of $\beta$ and its inverse become transparent, in view of \eqref{last3}.
\begin{itemize}
	\item[$\beta$:] For $\pi\in\a^{v}_{n+1,k}$, delete $\pi(2n+1)$ and $\pi(2n+2)$, then decrease the remaining letters larger than $\pi(2n+2)$ by $2$.
	\item[$\beta^{-1}$:] For $\sigma\in\a_{n,k}$, increase the letters no less than $\sigma(2n)$ by $2$, and append two letters $\sigma(2n)$ and $\sigma(2n)+1$ to the right of $\sigma$, in that order.
\end{itemize}
The proof ends here and we give the following example for illustration.
\end{proof}

\begin{example}
The two bijections $\alpha: \a^{p}_{n+1,k}\to \a_{n+1,k-1}$ and $\beta: \a^{v}_{n+1,k}\to \a_{n,k}$ for the case of $n=3$ are shown below.
\begin{align*}
\a^p_{4,3}\left\{
\begin{array}{c}
13254768 \\
13274658 \\
15243768 \\
17243658 \\
17263548
\end{array}
\right.\qquad & \xlongrightarrow{\alpha} \qquad \left.
\begin{array}{c}
13254867 \\
13284657 \\
15243867 \\
18243657 \\
18263547
\end{array}
\right\} \a_{4,2} \\
\a^p_{4,2}\left\{
\begin{array}{c}
13284657 \\
18243657 \\
18263547
\end{array}
\right.\qquad & \xlongrightarrow{\alpha} \qquad \left.
\begin{array}{c}
13284756 \\
18243756 \\
18273546
\end{array}
\right\} \a_{4,1} \\
\a^p_{4,1}\left\{
\begin{array}{c}
18273546
\end{array}
\right.\qquad & \xlongrightarrow{\alpha} \qquad \left.
\begin{array}{c}
18273645
\end{array}
\right\} \a_{4,0}\\
\a^v_{4,2}\left\{
\begin{array}{c}
13254867 \\
15243867
\end{array}
\right.\qquad & \xlongrightarrow{\beta} \qquad \left.
\begin{array}{c}
132546 \\
152436
\end{array}
\right\} \a_{3,2} \\
\a^v_{4,1}\left\{
\begin{array}{c}
13284756 \\
18243756
\end{array}
\right.\qquad & \xlongrightarrow{\beta} \qquad \left.
\begin{array}{c}
132645 \\
162435
\end{array}
\right\} \a_{3,1}\\
\a^v_{4,0}\left\{
\begin{array}{c}
18273645
\end{array}
\right.\qquad & \xlongrightarrow{\beta} \qquad \left.
\begin{array}{c}
162534
\end{array}
\right\} \a_{3,0}
\end{align*}

\end{example}






\subsection{\texorpdfstring{Other $\exc$-cases avoiding one pattern of length three}{}}
In this subsection we present the parallel $(-1)$-phenomena with respect to $\exc$, note the differences when one compares Table~\ref{tab:exc-avoid3} with Table~\ref{tab:avoid-3}. Unfortunately we have not found any $q$-analogues at this moment.

\begin{table}[tbp]\caption{The $(-1)$-evaluation over $\S_n(\tau)$ and $\D_n(\tau)$ with respect to $\exc$. The definition of the sequence $\{F_n\}$ is given in Conjecture \ref{conj}.}
\label{tab:exc-avoid3}
\centering
\begin{tabular}{ccccccc}
\hline
$\exc\setminus \tau$ &123 &132 &213 &231 &312 &321\\
\hline
$\S_{2n+1}$ &$\star$ &$(-1)^nC_n$ &$(-1)^nC_n$ &$\star$ &$\star$ &$(-1)^nC_n$ \\
$\D_{2n}$ &$(-1)^nF_n$ &$(-1)^nC_n$ &$(-1)^nC_n$ &$\star$ &$\star$ &$(-1)^nC_n$ \\
\hline
\end{tabular}
\end{table}

\begin{theorem}\label{nara-others}
For any $n\ge 1$,
\begin{align}
\label{132-odd}
\sum_{\pi\in\S_n(213)}(-1)^{\exc\:\pi}=\sum_{\pi\in\S_n(132)}(-1)^{\exc\:\pi}&=\begin{cases}0 & \emph{if $n$ is even},\\
(-1)^{\frac{n-1}{2}}C_{\frac{n-1}{2}} & \emph{if $n$ is odd},
\end{cases}\\
\label{132-even}
\sum_{\pi\in\D_n(213)}(-1)^{\exc\:\pi}=\sum_{\pi\in\D_n(132)}(-1)^{\exc\:\pi}&=\begin{cases}(-1)^{\frac{n}{2}}C_{\frac{n}{2}} & \emph{if $n$ is even},\\
0 & \emph{if $n$ is odd}.
\end{cases}
\end{align}
\end{theorem}
\begin{proof}
We first apply the $q=1$ case of Theorem~\ref{q-nara}, and the following identity due to Elizalde \cite{Eli} to derive the second equalities in both \eqref{132-odd} and \eqref{132-even}.
\begin{align}\label{321-132-exc-fix}
\sum_{\pi\in\S_n(321)}t^{\exc\:\pi}y^{\fix\:\pi}=\sum_{\pi\in\S_n(132)}t^{\exc\:\pi}y^{\fix\:\pi}, \quad \text{for $n\ge 1$}.
\end{align}
Next we observe the following facts, which can be easily checked.
\begin{align*}
\pi &\in\S_n(132) \Leftrightarrow \pi^{rc}\in\S_n(213),\\
\exc(\pi) &= n-\exc(\pi^{rc})-\fix(\pi),\: \fix(\pi)=\fix(\pi^{rc}).
\end{align*}
Consequently we have
\begin{align}\label{132-213-exc-fix}
\sum_{\pi\in\S_n(132)}t^{\exc\:\pi}y^{\fix\:\pi} &= t^n\sum_{\pi\in\S_n(213)}t^{-\exc\:\pi-\fix\:\pi}y^{\fix\:\pi}.
\end{align}
Plugging in $t=-1,y=0$ gives us directly the first equality in \eqref{132-even}. Finally, taking $t=y=-1$ in \eqref{132-213-exc-fix}, \eqref{321-132-exc-fix} and $t=-1,q=1$ in \eqref{wex-inv-gamma} leads to:
\begin{align*}
(-1)^n\sum_{\pi\in\S_n(213)}(-1)^{\exc\:\pi}=\sum_{\pi\in\S_n(321)}(-1)^{\wex\:\pi}=\begin{cases}0 & \textrm{if $n$ is even},\\
(-1)^{\frac{n+1}{2}}C_{\frac{n-1}{2}} & \textrm{if $n$ is odd},
\end{cases}
\end{align*}
which is exactly the first equality in \eqref{132-odd}.
\end{proof}


The only non-$\star$ entry in Table~\ref{tab:exc-avoid3} that is not covered by Theorems~\ref{q-nara} or \ref{nara-others} is still a conjecture.

\begin{conj}\label{conj}
For any $n\geq 1$, the polynomials $G_n(t):=\sum\limits_{\pi\in\D_n(123)}t^{\exc\:\pi}$ have nonnegative coefficients in their $\gamma$-expansions. Moreover, there is a sequence $\{F_n\}_{n\ge 1}$ of positive integers such that
\begin{align*}
G_{n}(-1)=\sum_{\pi\in\D_n(123)}(-1)^{\exc\:\pi}&=\begin{cases}(-1)^{\frac{n}{2}}F_{\frac{n}{2}} & \textrm{if $n$ is even},\\
0 & \textrm{if $n$ is odd}.
\end{cases}
\end{align*}
\end{conj}

We note that neither $\{G_n(1)\}_{n\ge 1}$ nor $\{F_n\}_{n\ge 1}$ is registered in the OEIS. The first values are given by $G_n(1)=0, 1, 2, 7, 20, 66, 218, 725,\ldots$ and
$F_n=1, 7, 58, 545, 5570, \ldots$.
For the first few $n\geq 1$, we have
\begin{align*}
G_1(t)&=0, G_2(t)=t,\\
G_3(t)&=t+t^2=t(1+t),\quad
G_4(t)=7t^2,\\
G_5(t)&=10t^2+10t^3=10t^2(1+t),\\
G_6(t)&=2t^2+62t^3+2t^4=2t^2(1+t)^2+58t^3,\\
G_7(t)&=109t^3+109t^4=109t^3(1+t),\\
G_8(t)&=45t^3+635t^4+45t^5=45t^3(1+t)^2+545t^4,\\
G_9(t)&=5t^3+1264t^4+1264t^5+5t^6=5t^3(1+t)^3+1249t^4(1+t),\\
G_{10}(t)&=769t^4+7108t^5+769t^6=769t^4(1+t)^2+5570t^5.
\end{align*}
The symmetry of $G_n(t)$ follows from the map $\pi\mapsto \pi^{rc}$, which is stable on $\S_n(123)$ and $\D_n(123)$, and satisfies  $\exc(\pi) = n-\exc(\pi^{rc})-\fix(\pi)$. Thus, if $\pi\in \D_n(123)$, we obtain the symmetry.

\section{Two cases avoiding two patterns of length four}\label{sec6: Sch}
We first enumerate $\A_n(2413,3142)$ and $\A_n(1342,2431)$, then put these results in the context of $(-1)$-evaluations of the descent polynomials over $\S_n(2413,3142)$ and $\S_n(1342,2431)$.

Letting $q=1$ in the last interpretation~\eqref{eq:231:132}
 of Theorem~\ref{thm:des-ai-qgamma} we derive immediately
the following $\gamma$-expansion for Narayana polynomials (see also \cite[Chapter 4]{Pet}).
\begin{align}
\label{gamma:nara}
C_n(t,1) &=\sum_{\pi\in \S_n(231)} t^{\des\:\pi}
=\sum_{k=0}^{\lfloor\frac{n-1}{2}\rfloor}\gamma_{n,k}^N t^k(1+t)^{n-1-2k},
\end{align}
where $\gamma_{n,k}^N:=\gamma_{n,k}(1)$.
The following two $\gamma$-expansions
 \eqref{gamma:sep} and \eqref{gamma:1342}, which  were  obtained recently by Fu-Lin-Zeng~\cite{FLZ} and  Lin\cite{Lin17}, respectively, will be crucial in our $(-1)$-evaluations.
 \begin{align}
 \label{gamma:sep}
S_n(t) &:=\sum_{\pi\in\S_n(2413,3142)}t^{\des\:\pi} =\sum_{k= 0}^{\lfloor\frac{n-1}{2}\rfloor}\gamma_{n,k}^S t^k (1+t)^{n-1-2k},\\
\label{gamma:1342}
Y_n(t) &:=\sum_{\pi\in\S_n(1342,2431)}t^{\des\:\pi} =\sum_{k= 0}^{\lfloor\frac{n-1}{2}\rfloor}\gamma_{n,k}^Y t^k (1+t)^{n-1-2k},
\end{align}
where
\begin{align}
\label{comb:sep}
\gamma_{n,k}^S &=\#\{\pi \in \S_n(2413, 3142): \dd^*\:\pi=0, \des\:\pi=k\},\\
\label{comb:1342}
\gamma_{n,k}^Y &=\#\{\pi \in \S_n(1342, 2431): \dd^*\:\pi=0, \des\:\pi=k\}.
\end{align}
It follows that
\begin{align}\label{alt-4=gamma}
|\A_n(2413,3142)|=\gamma_{n,\lfloor\frac{n-1}{2}\rfloor}^S,\; |\A_n(1342,2431)|=\gamma_{n,\lfloor\frac{n-1}{2}\rfloor}^Y.
\end{align}
Recall the $\gamma$-coefficients in the expansions \eqref{gamma:nara}, \eqref{gamma:sep}--\eqref{gamma:1342}. For $*=N,S,Y$, let
\begin{align*}
\Gamma_*(x,z)&:=\sum_{n=1}^{\infty}\sum_{k=0}^{\lfloor\frac{n-1}{2}\rfloor}\gamma^*_{n,k}x^kz^n
\end{align*}
be the generating functions for $\gamma^N_{n,k},\gamma^S_{n,k}$ and $\gamma^Y_{n,k}$, respectively. We need the following two algebraic equations for $\Gamma_S(x,z)$ and $\Gamma_Y(x,z)$, which were first derived by Lin \cite{Lin17}.
\begin{align}\label{ae-sep}
\Gamma_S&=z+z\Gamma_S+xz\Gamma^2_S+x\Gamma^3_S,\\
\label{ae-1342}
\Gamma_Y&=z+z\Gamma_Y+2xz\Gamma_N\Gamma_Y+x\Gamma^2_N(\Gamma_Y-z).
\end{align}

\subsection{The case of (2413,3142)--avoiding alternating permutations.}

\begin{theorem}\label{thm:alt-2413-odd}
Let $r_{n}:=|\A_{2n+1}(2413,3142)|,\:n\ge 0$, $R(x):=\sum\limits_{n=1}^{\infty}r_{n}x^n$, then
\begin{align}\label{sep-odd-gf}
R(x)=x(R(x)+1)^2+x(R(x)+1)^3.
\end{align}
Consequently, $r_0=1$ and for $n\ge 1$,
\begin{align}\label{sep-odd}
r_{n}=\dfrac{2}{n}\sum_{i=0}^{n-1}2^{i}\binom{2n}{i}\binom{n}{i+1}.
\end{align}
\end{theorem}
\begin{proof}
First, \eqref{alt-4=gamma} gives us $r_n=\gamma^S_{2n+1,n}$. Therefore, in order to get a recurrence relation for $r_n$, we should extract the coefficient of $z^{2n+1}$ in \eqref{ae-sep} and then compare the coefficients of $x^n$ from both sides. This gives us, for $n\ge 1$,
$$r_n=[x^{n-1}]\left([z^{2n}]\Gamma^2_S(x,z)\right)+[x^{n-1}]\left([z^{2n+1}]\Gamma^3_S(x,z)\right).$$
Now we take a closer look at $[z^{2n}]\Gamma^2_S(x,z)$.
$$[z^{2n}]\Gamma^2_S(x,z)=\sum_{m=1}^{2n-1}\left(\sum_{j=0}^{\lfloor\frac{m-1}{2}\rfloor}\gamma^S_{m,j}x^j\right)\cdot\left(\sum_{k=0}^{\lfloor\frac{2n-m-1}{2}\rfloor}\gamma^S_{2n-m,k}x^k\right),$$
So for each term in this summation, the power of $x$ is $$j+k\le\left\lfloor\frac{m-1}{2}\right\rfloor+\left\lfloor\frac{2n-m-1}{2}\right\rfloor\le n-1.$$
Hence we get contributions for $x^{n-1}$ only from odd $m$'s, with $j$ and $k$ both achieving their maxima. Similar analysis applies to the term involving $\Gamma^3_S$ and the details are omitted. All these amount to
$$r_n=\sum_{m=0}^{n-1}r_mr_{n-m-1}+\sum_{m,l=0}^{n-1}r_mr_lr_{n-m-l-1}.$$
In terms of the generating function $R(x)$, we obtain \eqref{sep-odd-gf}. Next we rewrite \eqref{sep-odd-gf} as
\begin{align}\label{sep-even-lag}
x=\frac{R}{(R+1)^2(R+2)},
\end{align}
which is ripe for applying the Lagrange inversion (cf. \cite{FSe}). A straightforward computation leads to \eqref{sep-odd} and completes the proof.
\end{proof}

\begin{theorem}
Let $t_n:=|\A_{2n}(2413,3142)|,\:n\ge 1$, $T(x):=\sum\limits_{n=1}^{\infty}t_nx^n$, then
\begin{align}\label{sep-even-gf}
\frac{1}{2}R(x)=\frac{1}{2}R(x)\cdot T(x)+T(x).
\end{align}
Consequently, $t_1=1$ and for $n\ge 2$,
\begin{align}\label{sep-even}
t_{n}=\dfrac{4}{n-1}\sum_{i=0}^{n-2}2^{i}\binom{2n-1}{i}\binom{n-1}{i+1}.
\end{align}
\end{theorem}

\begin{proof}
It may still be possible to establish \eqref{sep-even-gf} using the algebraic equation \eqref{ae-sep}, but this time we present a combinatorial argument, showing both sides generate the same set of permutations.

The first thing to notice is that for an alternating permutation $\pi\in\A_{2n+1}(2413,3142)$, $n\ge 1$, its reverse $\pi^r\neq \pi$ is also in $\A_{2n+1}(2413,3142)$. This implies that $r_n$ is even for $n\ge 1$. Moreover, we call a permutation $\pi\in\S_n, n\ge 2$ {\em normal} if $1$ appears to the left of $n$. For example, there are three normal permutations in $\S_3$: $213,123,132$. Now we see that exactly one permutation in the pair $\{\pi,\pi^r\}$ is normal, and consequently the number of normal permutations in $\A_{2n+1}(2413,3142)$ is $r_n\slash 2$. Therefore the left-hand side of \eqref{sep-even-gf} generates all normal, alternating, and $(2413,3142)$-avoiding permutations of odd length larger than $1$. Next we show that the right-hand side does precisely the same. The following claimed decomposition, whose proof given separately, is the key ingredient. Recall two classical operations, the \emph{direct sum} ``$\oplus$'' and the \emph{skew sum} ``$\ominus$''. If $\pi=\pi^{(1)}\pi^{(2)}$ with $\pi^{(1)}<\pi^{(2)}$, then we write $\pi=\pi^{(1)}\oplus\st(\pi^{(2)})$. Similarly, if $\pi=\pi^{(1)}\pi^{(2)}$ with $\pi^{(1)}>\pi^{(2)}$, then we write $\pi=\st(\pi^{(1)})\ominus\pi^{(2)}$. For instance, we have $12354=123\oplus 21$ and $34521=123\ominus 21$.
\begin{claim}\label{claim}
Let $\pi$ be a normal, alternating, and $(2413,3142)$-avoiding permutation of odd length larger than $1$, then there exists a unique pair of permutations $(\pi^{(1)},\pi^{(2)})$, such that
\begin{enumerate}
    \item $\pi=\pi^{(1)}\oplus\pi^{(2)}$,
    \item either $\pi^{(1)}=1$ or $\pi^{(1)}$ is of odd length and non-normal, alternating and $(2413,3142)$-avoiding,
    \item $\pi^{(2)}$ is of even length ($\ge 2$) and $(2413,3142)$-avoiding, its reverse is alternating.
\end{enumerate}
\end{claim}
In view of the claim above, $\frac{1}{2}R(x)\cdot T(x)$ accounts for the cases when $\pi^{(1)}$ is of length $3$ or longer, while $T(x)$ corresponds to the case when $\pi^{(1)}=1$. Now since the above decomposition using $\oplus$ is unique, we get \eqref{sep-even-gf}.

Applying \eqref{sep-even-lag}, we can rewrite \eqref{sep-even-gf} as
\begin{align*}
T=\frac{R}{R+2}=x(R+1)^2.
\end{align*}
This form is suitable for the more general Lagrange-B\"urmann formula (cf. \cite{FSe}), and we get for $n\ge 2$,
\begin{align*}
t_n&=[x^{n-1}](R+1)^2=\frac{1}{n-1}[R^{n-2}]\left(2(R+1)(R+1)^{2n-2}(R+2)^{n-1}\right)\\
&=\frac{2}{n-1}\sum_{i=0}^{n-2}2^{n-1-i}\binom{n-1}{i}\binom{2n-1}{n-2-i}\\
&=\frac{4}{n-1}\sum_{i=0}^{n-2}2^i\binom{n-1}{i+1}\binom{2n-1}{i}.
\end{align*}
The proof is now completed.
\end{proof}

\begin{proof}[Proof of Claim \ref{claim}]
A permutation avoids both $2413$ and $3142$ if and only if it is \emph{separable} (cf. \cite[page 57]{Kit}), which means it can be decomposed as either $\pi=\pi^{(1)}\oplus\pi^{(2)}$ or $\pi=\pi^{(1)}\ominus\pi^{(2)}$. Now $\pi$ being normal excludes the case of $\pi=\pi^{(1)}\ominus\pi^{(2)}$. Such $\oplus$-decomposition may not be unique. To make it unique as claimed, we always take the decomposition where $\pi^{(1)}$ is shortest in length. This also means $\pi^{(1)}$ itself cannot be $\oplus$-decomposed further. Therefore $\pi^{(1)}=1$ or $\pi^{(1)}$ can be $\ominus$-decomposed and thus non-normal. Being subwords of $\pi$, $\pi^{(1)}$ and $\pi^{(2)}$ should avoid $2413$ and $3142$ as well. The remaining restrictions on $\pi^{(1)}$ and $\pi^{(2)}$ can be verified easily.
\end{proof}

\begin{remark}
In view of the similarity in the expressions for $r_n$ and $t_n$, we can unify them as the following formula:
\begin{align*}
|\A_n(2413,3142)|=\frac{2^{n-2m}}{m}\sum_{i=0}^{m-1}2^i\binom{m}{i+1}\binom{n-1}{i}, \text{ where }m=\left\lfloor\frac{n-1}{2}\right\rfloor, \text{ and }n\ge 3.
\end{align*}
Moreover, the two sequences $\{r_n\}_{n\ge 0}$ and $\{t_n\}_{n\ge 1}$ have been cataloged in the OEIS (see \href{https://oeis.org/A027307}{\tt{oeis:A027307}} and \href{https://oeis.org/A032349}{\tt{oeis:A032349}}), and were considered, for instance, by Deutsch et al. \cite{DC} as enumerating certain type of lattice paths. Then a natural question would be to find a bijection between these two combinatorial models.
\end{remark}

Now we turn to the $(-1)$-evaluation for $S_n(t)$, which is a direct result of \eqref{gamma:sep} and \eqref{alt-4=gamma}. 

\begin{theorem}\label{sep}
For any $n\ge 1$, there holds
\begin{align}
\label{alt:sep-odd}
S_{n}(-1)=\sum_{\pi\in\S_n(2413,3142)}(-1)^{\des\:\pi}&=\begin{cases}0 & \emph{if $n$ is even},\\
(-1)^{\frac{n-1}{2}}r_{\frac{n-1}{2}} & \emph{if $n$ is odd}.
\end{cases}
\end{align}
\end{theorem}

\subsection{The case of (1342,2431)--avoiding alternating permutations.}
\begin{theorem}\label{case1342-2431}
Let $u_{n}:=|\A_{2n+1}(1342,2431)|$ and  $U(x):=\sum\limits_{n=0}^{\infty}u_{n}x^n$, then
\begin{align}\label{1342-odd-gf}
U(x)&=\frac{\sqrt{1-4x}}{\sqrt{1-4x}-2x}=\cfrac{1}{1-\cfrac{2x}{1-\cfrac{2x}{1-\cfrac{x}{1-\cfrac{x}{\ddots}}}}}.
\end{align}
\end{theorem}
\begin{proof}
We only sketch the proof since it is quite analogous to that of Theorem~\ref{thm:alt-2413-odd}. We use \eqref{ae-1342} in a similar way as we use \eqref{ae-sep} in the proof of \eqref{sep-odd-gf}, i.e., we extract the coefficients of $z^{2n+1}$ from both sides and then compare the coefficients of $x^{n}$. This leads to the following recurrence relation that involves the Catalan number $C_n$, since we have already shown that $\gamma_{2n+1,n}^N=|\A_{2n+1}(231)|=C_n$. For $n\ge 1$, we have:
\begin{align*}
u_n&=2\sum_{m=0}^{n-1}u_mC_{n-1-m}+\sum_{m=1}^{n-1}u_m\sum_{l=0}^{n-m-1}C_lC_{n-m-l-1}\\
&=2\sum_{m=0}^{n-1}u_mC_{n-1-m}+\sum_{m=1}^{n-1}u_mC_{n-m}.
\end{align*}
In terms of generating functions, this means
\begin{align*}
2(U(x)-1)=2xU(x)C(x)+(U(x)-1)C(x),
\end{align*}
where
\begin{align*}
C(x)=\frac{1-\sqrt{1-4x}}{2x}
\end{align*}
is the generating function for the Catalan numbers. We plug in $C(x)$ and solve for $U(x)$ to finish the proof.
\end{proof}

\begin{remark}
Interestingly, our result above seems to be the first combinatorial interpretation for $u_n$, and the sequence $\{u_n\}_{n\ge 0}$ is also on OEIS (see \href{https://oeis.org/A084868}{\tt{oeis:A084868}}).
Although a single sum formula for  $u_n$ can be derived from  \eqref{1342-odd-gf}
by using standard method, we prefer to give  a multiple sum formula  as follows:
\begin{align*}
\sum_{n=0}^{\infty}u_n x^n&=\dfrac{1}{1-\dfrac{2x}{\sqrt{1-4x}}}
=\sum_{m=0}^{\infty} \left(\frac{2x}{\sqrt{1-4x}}\right)^m\\
&=\sum_{m=0}^{\infty} \left(\sum_{k=0}^{\infty} 2\binom{2k}{k} x^{k+1}\right)^m.
\end{align*}
Thus we obtain, for $n\geq 1$,
\begin{equation}\label{sum formula}
u_n=\sum_{m=1}^n 2^m \sum_{k_1+\cdots +k_m=n-m}\prod_{i=1}^m \binom{2k_i}{k_i}.
\end{equation}
The above formula shows that  $u_{n}$ is a multiple of 4 when $n\geq 2$.
\end{remark}

With the aid of \eqref{gamma:1342} and \eqref{alt-4=gamma}, we obtain
\begin{theorem}
For any $n\ge 1$, there holds
\begin{align}
\label{alt:1342}
Y_{n}(-1)=\sum_{\pi\in\S_n(1342,2413)}(-1)^{\des\:\pi}&=\begin{cases}0 & \emph{if $n$ is even},\\
(-1)^{\frac{n-1}{2}}u_{\frac{n-1}{2}} & \emph{if $n$ is odd}.
\end{cases}
\end{align}
\end{theorem}

We end this section by noting that both $\S_n(2413,3142)$ and $\S_n(1342,2431)$ exhibit $(-1)$-phenomenon only for the entire set of permutations, but not for the subset of coderangements. This should not come as a surprise in view of the reversal relations between the two patterns that we avoid, namely $(2413)^r=3142, (1342)^r=2431$, and the fact that the definition of coderangements is incompatible with the reverse map. Other subsets of $\S_n$ instead of $\D^*_n$ should be examined to hunt for the other half of the $(-1)$-phenomenon.

\section{Final remarks}
It would be interesting to give direct combinatorial proofs of the $(-1)$-phenomena of this paper. The expansions we have in Theorems~\ref{thm:des-ai-qgamma} and \ref{thm:des-ai-qgamma-new}, Lemma~\ref{lem:132-fl} are all natural, in the sense that the statistics (powers of $q$) appear in the $\gamma$-coefficients on the expansion side, are the same as those appear on the left-hand side, the avoiding patterns are also the same. And we prove them uniformly using the MFS-action and its variation. Each interpretation listed in Theorem~\ref{qnara} (resp. Theorem~\ref{HZthwex}) should have a $q$-$\gamma$-expansion in theory. Namely, once we have an expansion for one of them, the others all share this expansion. But expansions derived this way are unnatural (for instance, \eqref{wex-inv-gamma} is unnatural).  So now the question is, do the other ones that we are missing in Theorem~\ref{thm:des-ai-qgamma} (to be precise, \#7---\#10 in Tabel~\ref{ten}) have natural expansions? It seems the MFS-action cannot help anymore.

It would be appealing to establish a multivariate generating function (in the spirit of Shin-Zeng's Lemma~\ref{Lemma:Shin-Zeng}) that specializes to the $(2413, 3142)$-avoiding permutations or $(1342, 2413)$-avoiding permutations, and consequently giving us $q$-analogues of \eqref{alt:sep-odd} or \eqref{alt:1342}.


Another direction to extend the results presented here is to place $\S_n$ in the broader context of Coxeter groups, and consider the so-called  Narayana polynomials of types B and D (see \cite[Theorems 2.32 and 2.33]{Ath}). This approach was shown fruitful for permutations in a recent work of Eu et al.\cite{EFHL}.

Finally, in a different context,
 some $(-1)$-phenomenon have been generalized to the deeper \emph{cyclic sieving phenomenon} (CSP), see \cite{RSW, Sag}.
It would be interesting to see whether there are any
CSP-analogue for our $(-1)$-phenomenon.

\section*{Acknowledgement}
We are indebted to the referees for their professionalness and meticulosity in reviewing our manuscript. Their suggestions have optimized the exposition of this paper and enhanced its readability.

The first and second authors were supported by the National Science Foundation of China (No.~11501061) and the Fundamental Research Funds for the Central Universities (No.~2018CDXYST0024). The third author was supported by the China Scholarship Council.

\end{document}